\documentclass[AMS]{WileyNJD-v1}

\articletype{Article Type}%

\usepackage{amssymb,amsmath,amsthm}
\usepackage{mathtools}
\usepackage{float}
\usepackage{xifthen}
\usepackage{graphicx}

\usepackage{tikz,pgfplots}
\pgfplotsset{compat=1.13}

\usepackage[nameinlink]{cleveref}

\usepackage{ifpdf}

\hypersetup{
    colorlinks=true,
    pdfborder={0 0 0},
}

\newcommand{\bbR}{\mathbb{R}}

\newcommand{\bbZ}{\mathbb{Z}}

\newcommand{\bbC}{\mathbb{C}}
\newcommand{\dyad}{\mathbf{\otimes}}

\newcommand{\abs}[1]{\left\lvert #1 \right\rvert}

\newcommand{\del}{\boldsymbol{\nabla}}

\newcommand{\leftcr}{\left\{}
\newcommand{\rightcr}{\right\}}

\newcommand{\lininter}[1]{{\cal{I}}_h \left[#1 \right]}

\newcommand{\Ltwo}[1]{%
\ifthenelse{\equal{#1}{}}{L^2}{L^2(#1)}%
}

\newcommand{\Ltwoz}[1]{%
\ifthenelse{\equal{#1}{}}{L^2_0}{L^2_0(#1)}%
}

\newcommand{\Cone}[1]{%
\ifthenelse{\equal{#1}{}}{C^{1}}{C^{1}(#1)}%
}

\newcommand{\Conez}[1]{%
\ifthenelse{\equal{#1}{}}{C^{1}_{0}}{C^{1}_{0}(#1)}%
}

\newcommand{\Ctwo}[1]{%
\ifthenelse{\equal{#1}{}}{C^{2}}{C^2(#1)}%
}

\newcommand{\Ctwoz}[1]{%
\ifthenelse{\equal{#1}{}}{C^{2}_{0}}{C^{2}_{0}(#1)}%
}

\newcommand{\Cholder}[1]{%
\ifthenelse{\equal{#1}{}}{C^{0,\gamma}}{C^{0,\gamma}(#1)}%
}

\newcommand{\Cholderz}[1]{%
\ifthenelse{\equal{#1}{}}{C^{0,\gamma}_{0}}{C^{0,\gamma}_{0}(#1)}%
}

\newcommand{\Hone}[1]{%
\ifthenelse{\equal{#1}{}}{H^1}{H^1(#1)}%
}

\newcommand{\Honez}[1]{%
\ifthenelse{\equal{#1}{}}{H^1_0}{H^1_0(#1)}%
}

\newcommand{\Htwo}[1]{%
\ifthenelse{\equal{#1}{}}{H^2}{H^2(#1)}%
}

\newcommand{\Htwoz}[1]{%
\ifthenelse{\equal{#1}{}}{H^2_0}{H^2_0(#1)}%
}

\newcommand{\bolds}[1]{\boldsymbol{#1}}

\newcommand{\bnu}{\bolds{\nu}}

\ifpdf
  \DeclareGraphicsExtensions{.eps,.pdf,.png,.jpg}
\else
  \DeclareGraphicsExtensions{.eps}
\fi

\begin{document}

\title{Numerical convergence of nonlinear nonlocal continuum models to local elastodynamics}

\author[1]{Prashant K. Jha}

\author[2]{Robert Lipton*}

\authormark{Prashant K. Jha and Robert Lipton}

\address[1]{\orgdiv{Department of Mathematics}, \orgname{Louisiana State University}, \orgaddress{\state{LA}, \country{USA}}}

\address[2]{\orgdiv{Department of Mathematics}, \orgname{Louisiana State University}, \orgaddress{\state{LA}, \country{USA}}}

\corres{*Robert Lipton, 384 Lockett Hall, LSU, Baton Rouge. \email{lipton@math.lsu.edu}}


\abstract{We quantify the numerical error and modeling error associated with replacing a nonlinear nonlocal bond-based peridynamic model with a local elasticity model or a linearized peridynamics model away from the fracture set. The nonlocal model treated here is characterized by a double well potential and is a smooth version of the peridynamic model introduced in \cite{CMPer-Silling}.  The solutions of nonlinear peridynamics are shown to converge to the solution of linear elastodynamics at a rate linear with respect to the length scale $\epsilon$  of non local interaction. This rate also holds for the convergence of solutions of the linearized peridynamic model to the solution of the local elastodynamic model. For local linear Lagrange interpolation the consistency error for the numerical approximation is found to depend on the ratio between mesh size $h$ and $\epsilon$. More generally for local Lagrange interpolation of order $p\geq 1$ the consistency error is of order $h^p/\epsilon$.  A new stability theory for the time discretization is provided and an explicit generalization of the CFL condition on the time step and its relation to mesh size $h$ is given. Numerical simulations are provided illustrating  the consistency error associated with the convergence of nonlinear and linearized peridynamics to linear elastodynamics. 
}

\jnlcitation{\cname{%
\author{Prashant K. Jha}, 
\author{Robert Lipton}, (\cyear{2018}), 
\ctitle{Numerical convergence of nonlinear nonlocal continuum models to local elastodynamics}, \cjournal{International Journal for Numerical Methods in Engineering}, \cvol{vol 114(13), pages 1389-1410, https://doi.org/10.1002/nme.5791}.}}

\maketitle


%

\section{Introduction}
The nonlocal formulation proposed in \cite{CMPer-Silling} provides a framework for modeling crack propagation inside solids. The basic idea is to redefine the strain in terms of the difference quotients of the displacement field and allow for nonlocal forces acting within a finite horizon. The relative size of the horizon with respect to the diameter of the domain of the specimen is denoted by $\epsilon$. The force at any given material point is determined by the deformation of all neighboring material points surrounding it within a radius given by the size of horizon. 
Computational fracture modeling using peridynamics feature formation and evolution of interfaces associated with fracture, see \cite{CMPer-Silling5}, \cite{BobaruHu}, \cite{HaBobaru}, \cite{CMPer-Agwai}, \cite{CMPer-Ghajari},  \cite{Diehl}, \cite{CMPer-Bobaru}, \cite{CMPer-Dayal}, \cite{CMPer-Silling7}, \cite{CMPer-Bobaru2}, \cite{SillBob}, \cite{CMPer-Silling5}, \cite{WeckAbe}, \cite{CMPer-Silling8}, and \cite{CMPer-Lipton2}. Theoretical analysis of different mechanical and mathematical aspects of peridynamic models can be found in \cite{CMPer-Silling7},  \cite{CMPer-Du3}, \cite{CMPer-Du5}, \cite{CMPer-Emmrich}, \cite{AksoyluParks}, \cite{AksoyluUnlu}, \cite{CMPer-Dayal}, and\cite{CMPer-Dayal2}. 
A full accounting of the peridynamics literature lies beyond the scope of this paper however several themes and applications are covered in the recent handbook \cite{Handbook}.

In the absence of fracture, earlier work demonstrates the convergence  of linear peridynamic models to the local model of linear elasticity as $\epsilon$ goes to zero, see \cite{CMPer-Weckner}, \cite{CMPer-Silling4}. The convergence of an equilibrium peridynamic model to the Navier equation in the sense of solution operators is established in \cite{CMPer-Mengesha}. Numerical analysis of linear peridynamic models for 1-d bars have been given in \cite{CMPer-Weckner} and \cite{CMPer-Bobaru}.  Related approximations of nonlocal diffusion models are discussed in \cite{CMPer-Du4}, \cite{CMPer-Chen}, and \cite{CMPer-Du1}.  A stability analysis of the numerical approximation to solutions of linear nonlocal wave equations is given in \cite{CMPer-Guan}.

In this work we analyze the discrete approximations to the nonlinear nonlocal model developed in \cite{CMPer-Lipton3},  \cite{CMPer-Lipton}. This model is a smooth version of the prototypical micro-elastic model introduced in \cite{CMPer-Silling}, see \cref{s:intro}.  
In earlier theoretical work, it has been shown that in the limit of vanishing non locality this model delivers evolutions possessing sharp displacement discontinuities associated with cracks. The limiting displacement field evolution has bounded Griffith fracture energy  and away from the fracture set satisfies classic local elastodynamics \cite{CMPer-Lipton3},  \cite{CMPer-Lipton}, \cite{CMPer-JhaLipton}.  This model motivates adaptive implementations of peridynamics for brittle fracture. In regions of the body where where brittle fracture is anticipated one would apply the nonlinear nonlocal model but in regions where no fracture is to be anticipated one would like to apply the linear elastic model.  In this paper we will assume the solution is differentiable and there is no fracture. Here we investigate the difference between numerically computed solutions for the nonlinear nonlocal bond based model with those of  the linearized nonlocal model, and  those of classic local elastodynamics.  The types of nonlocal kernels  associated with these prototypical models are central to the theory but up till now have not been treated in the literature.

In this work we show that the solutions of the nonlinear model converge to classical elastodynamics at a rate that is linear in $\epsilon$. We analyze the numerical approximation associated with linear interpolation in space for two cases: i.) when the size of horizon is fixed and the mesh size $h$ tends to zero, known as $h$-convergence, and ii.) when the size of the horizon also tends to zero and the mesh approaches zero faster than the horizon. 
For the first case we show consistency error is of order $O(\frac{h}{\epsilon})$ for both nonlinear and linearized models, see \cref{prop:3.1}. For the second we find that the consistency error for both models is $O(\frac{h}{\epsilon})+O(\epsilon)$, see \cref{prop:3.2}. 
These ideas are easily extended to higher order local Lagrange interpolation. For $p^{th}$ order local polynomial interpolations $p\geq 1$ the consistency error for both models and  case i.) is of order $O(h^p/\epsilon)$ and for both models and case ii.) is of order $O(h^p/\epsilon) + O(\epsilon)$, see \cref{prop:3.3} and \cref{prop:3.4}. These results show that the grid refinement relative to the horizon length scale has
more importance than decreasing the horizon length when establishing convergence to the
classical elastodynamics description.

Earlier related work \cite{CMPer-JhaLipton}  analyzes the nonlinear model and establishes the existence of non-differentiable H\"older continuous solutions. It is shown there that the rate of convergence of the discrete model to the continuum nonlocal model is of the order $h^\gamma/\epsilon$ where $0<\gamma\leq 1$ is the H\'older exponent.
The work presented here shows that we can improve the rate of convergence for this model if we have a-priori knowledge on the number of bounded continuous derivatives of the solution.
In this paper we have restricted the analysis and simulations to the one dimensional case to illustrate the ideas.   For higher dimensional problems the convergence rates  are the same, see \cref{ss: higherD}, and future work will address the consistency error in higher dimensions using the same techniques developed here.

A second issue is the coordination of spatial and temporal discretization to insure stability for numerical approximation of nonlocal models. Here the stability for the central difference in time approximation to the linearized model is considered. Analysis of the linearized peridynamic nonlocal model shows that the stability is given by a new explicit condition that converges to the well known CFL condition as $\epsilon\rightarrow 0$, see \cref{thm:4.2}. One no longer has an explicit stability condition for the non-linear model. However it is found that the semi-discrete approximation of the nonlinear model is stable in the energy norm, see  \cite{CMPer-JhaLipton}.

In \cref{s:numerical} we present numerical simulations that confirm the error estimates for both linearized and non-linear peridynamics. The numerical experiments show that the discretization error can be reduced by choosing the ratio $h/\epsilon$ suitably small for every choice of $\epsilon$ as $\epsilon\rightarrow 0$, see \cref{fig:error plot m converg}.  We verify the  convergence rates by simulating the peridynamic model long enough to include the boundary effects due to wave reflection in \cref{ss:nondim}. Our numerical studies confirm that the solutions of linear and nonlinear peridynamics are indistinguishable for sufficiently small horizon $\epsilon$.

The organization of this article is as follows: In \cref{s:intro}, we introduce the class of nonlocal nonlinear potentials and describe the convergence of peridynamic models to classical elastodynamics. In \cref{s:fe discr}, we introduce the finite element approximation of the model and present bounds on the discretization error. In \cref{s:time discr}, we consider the central  difference in time scheme and obtain the stability condition on $\Delta t$ as function of $\epsilon$ and $h$. In \cref{s:numerical}, we present the numerical simulations. In \cref{ss: higherD} we present the convergence of the model in higher dimensions. The proofs of the theorems are given in \cref{s:proof} and we provide our conclusions in \cref{s:concl}.

\section{Nonlocal evolution and elastodynamics}\label{s:intro}
The mathematical formulation for the nonlocal model is presented in this section. We  exhibit the convergence rate of nonlocal solutions to the solution to linear elastodynamics in the limit of vanishing peridynamic horizon.  A convergence rate is also provided for the linearized nonlocal model. The convergence rate for the nonlocal kernels treated here have not been addressed before in the literature.
\subsection{The nonlocal model}

We consider the nonlocal potentials introduced in \cite{CMPer-Lipton,CMPer-Lipton3}. 
Let $D := [a,b] \subset \bbR$ be a bounded material domain in one dimension and $J = [0,T]$ be an interval of time. The nonlocal boundary denoted by $\partial D^\epsilon$ are intervals of diameter $2\epsilon$ on either side of $D$ and  given by $(a-\epsilon,a+\epsilon)\cup(b-\epsilon,b+\epsilon)$.  The strain $S$ for the one dimensional peridynamic model is given by the difference quotient 
\begin{align*}
S(y,x;u) &:= \dfrac{u(y) - u(x)}{|y-x|}.
\end{align*}
The nonlocal force is given in terms of the non-linear two-point interaction potential $W^{\epsilon}$ defined by
\begin{align*}
W^{\epsilon}(S, y - x) &= \dfrac{2J^{\epsilon}(\abs{y - x})}{\epsilon\abs{y - x}} f(\abs{y - x} S^2),
\end{align*}
where $f: r\in\bbR^{+} \to \bbR$ is positive, smooth, and concave with following properties
\begin{equation}\label{eq:per asymptote}
\lim_{r\to 0^+} \frac{f(r)}{r} = f'(0) \qquad \text{ and } \qquad \lim_{r\to \infty} f(r) = f_{\infty} < \infty. 
\end{equation}
The potential $W^{\epsilon}(S, y - x)$ is of double well type and convex near the origin where it has one well the second well is at $\infty$ and associated with the horizontal asymptote $W^\epsilon(\infty,y - x)$, see \cref{fig:per pot}. The function $J^{\epsilon}(\abs{y - x})$  influences the magnitude of the nonlocal force due to $y$ on $x$. We define $J^{\epsilon}$ by rescaling $J(\abs{\xi})$, i.e. $J^{\epsilon}(\abs{\xi}) = J(\abs{\xi}/\epsilon)$. The influence function $J$ is zero outside the ball $[-1,1]$, and satisfies $0\leq J(\abs{\xi}) \leq M$ for all $\xi \in [-1,1]$. 

The force of two point interaction between $x$ and $y$ is derived from the nonlocal potential and given by $\partial_S W^\epsilon(S,y - x)$, see \cref{fig:first der per pot}.
For small strains the force is linear and elastic and then softens as the strain becomes larger. The critical strain, for which the force between $x$ and $y$ begins to soften, is given by $S_c(y, x) := \bar{r}/\sqrt{\abs{y - x}}$ and the force decreases monotonically for
\begin{align*}
\abs{S(y, x; u)} > S_c.
\end{align*}
Here $\bar{r}$ is the inflection point of $r:\to f(r^2)$, and is the root of following equation
\begin{align*}
f'({r}^2) + 2{r}^2 f''({r}^2) = 0.
\end{align*}

\begin{figure}
\centering
\includegraphics[scale=0.25]{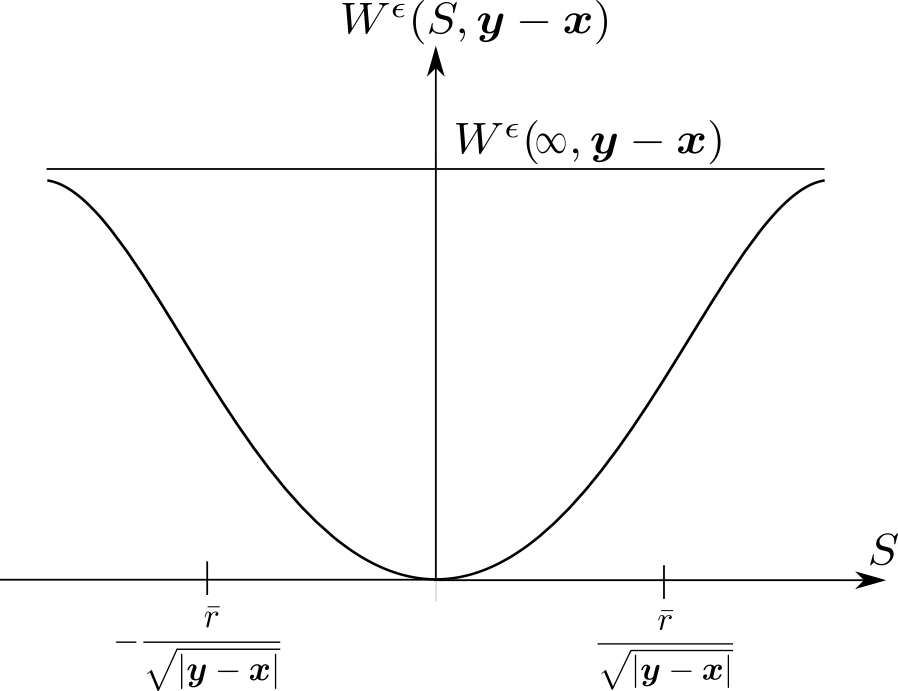}
\caption{Two-point potential $W^\epsilon(S,y - x)$ as a function of strain $S$ for fixed $y - x$.}
 \label{fig:per pot}
\end{figure}  

\begin{figure}
\centering
\includegraphics[scale=0.25]{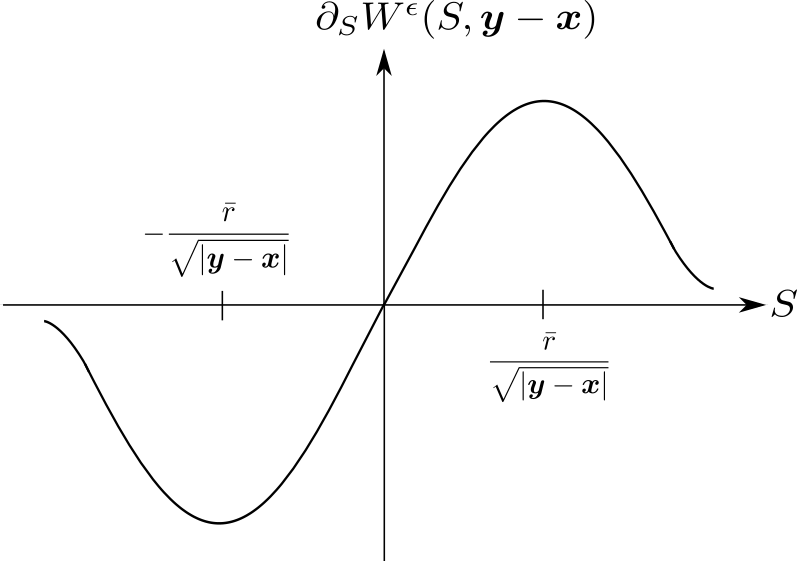}
\caption{Nonlocal force $\partial_S W^\epsilon(S,y - x)$ as a function of strain $S$ for fixed $y - x$. Second derivative of $W^\epsilon(S,y-x)$ is zero for $S=S_c:=\pm \bar{r}/\sqrt{|y -x|}$.}
 \label{fig:first der per pot}
\end{figure}

\begin{figure}
\centering
\includegraphics[scale=0.25]{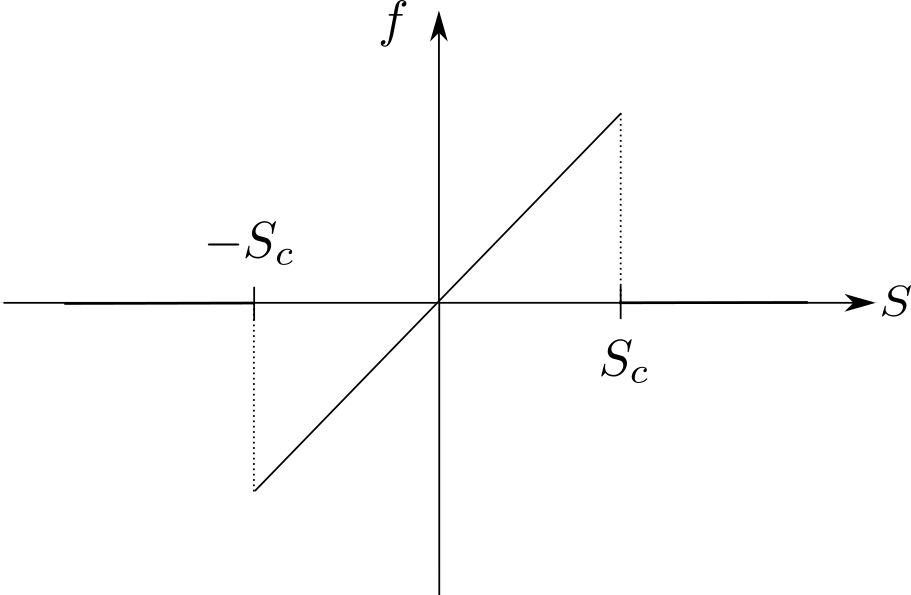}
\caption{Prototypical micro-elastic bond model of \cite{CMPer-Silling} as a function of strain $S$ for fixed $y - x$. Here the nonlocal force drops to zero at $S=S_c$.}
 \label{fig:first der per pot pmb}
\end{figure}  
The nonlocal force
$-\del PD^\epsilon$ is defined by
\begin{align*}
-\del PD^\epsilon(u)(x) &=\dfrac{1}{2\epsilon} \int_{x-\epsilon}^{x+\epsilon}\partial_S W^\epsilon(S,y - x)dy \\
&= \dfrac{2}{\epsilon^2} \int_{x-\epsilon}^{x+\epsilon} J(|y-x|/\epsilon) f'(|y-x| S(y,x;u)^2) S(y,x;u)dy.
\end{align*}
This force-strain model is a smooth version of the prototypical micro elastic model \cite{CMPer-Silling} which exhibits an abrupt drop in the force after a critical strain, see \cref{fig:first der per pot pmb}.

Similarly, we denote $-\del PD^\epsilon_l(u)(x)$ as the linearized  peridynamic force at $x$, given by
\begin{align*}
-\del PD^\epsilon_l(u)(x) &= \dfrac{2}{\epsilon^2} \int_{x-\epsilon}^{x+\epsilon} J(|y-x|/\epsilon) f'(0) S(y,x;u)dy.
\end{align*}
The corresponding linearized local model is characterized by the Young's modulus $\bbC$ given by
\begin{align}
\bbC &= \int_{-1}^{1} J(|z|) f'(0) |z| dz \label{eq:elasto const C} \\
&= \dfrac{1}{\epsilon^2} \int_{x-\epsilon}^{x+\epsilon} J(|y-x|/\epsilon)f'(0) |y-x| dy, \qquad \forall x, \epsilon>0.\notag
\end{align}

\subsection{The dynamic evolution}
We now state the initial boundary problem for three the types of evolutions: the first is given by the nonlinear nonlocal  model, the second given by the linearized nonlocal  model, and  the third given by the classic local linear elastic model. 
Let $u^\epsilon$ be the solution of the peridynamic equation of evolution, $u^\epsilon_l$ be the solution of the linearized peridynamic equation of evolution, and $u$ be the solution of elastodynamic equation of evolution with Young's modulus $\mathbb{C}$. For comparison of $u^\epsilon_l$ and $u^\epsilon$ with $u$, we assume $u$ to be extended by zero outside $D$. The displacements $u^\epsilon$, $u^\epsilon_l$, and $u$ satisfy following evolution equations, for all $(x,t) \in D\times J$, described by
\begin{align}
\rho\ddot{u}(t,x) &= \bbC u_{xx}(t,x) + b(t,x), \label{eq:elasto} \\
\rho\ddot{u}^\epsilon(t,x) &= -\del PD^\epsilon(u^\epsilon(t))(x) + b(t,x), \label{eq:peri} \\
\rho \ddot{u}^\epsilon_l(t,x) &= -\del PD^\epsilon_l(u^\epsilon_l(t))(x) + b(t,x), \label{eq:peri linear}
\end{align}
where $b(t,x)$ is a prescribed body force and the mass density $\rho$ is taken to be constant.
The boundary conditions are given by 
\begin{align*}
u^\epsilon(t,x) = 0 \qquad \text{and} \qquad \dot{u}^\epsilon(t,x) = 0, \qquad \forall t\in J, \forall x\in \partial D^\epsilon,
\end{align*}
and the same boundary conditions hold for $u^\epsilon_l$ and $u$.
The initial condition is given by
\begin{align*}
u^\epsilon(0,x) = g(x) \qquad \text{and} \qquad \dot{u}^\epsilon(0,x) = h(x), \qquad \forall x\in D,
\end{align*}
with $g=h=0$ outside some fixed subset $D'$ of $D$. The same initial condition also holds for $u^\epsilon_l$ and $u$. For future reference we denote the width of the layer $D\setminus D'$ by $\delta$.

\subsection{Convergence of nonlocal models in the limit of vanishing horizon}
In this section we provide convergence rates that show that the solution $u^\epsilon$ of the peridynamic equation converges, in the limit $\epsilon \to 0$, to the solution $u$ of the elastodynamic equation. The model treated here was considered earlier but for solutions that may not be differentiable and exhibit discontinuities \cite{CMPer-Lipton,CMPer-Lipton3}. Convergence was established for this case, however no convergence rate is available.  For linear nonlocal models with kernels different than the ones treated here, the limiting behavior has been identified in by several investigators in the peridynamics literature, see \cite{CMPer-Emmrich2,CMPer-Silling4,CMPer-Emmrich,AksoyluUnlu}.

We first provide estimates for the difference between the peridynamics force, the linearized peridynamics force, and the elastodynamics force. With these estimates in hand we are then present the rate of convergence of the solution of the nonlinear nonlocal evolution to the solution of the local linear elastic wave equation. In what follows $C^n(D)$ is the space of functions with $n$ continuous derivatives on $D$. 

{\vskip 2mm}
\begin{proposition}[Control on the difference between peridynamic force and local elastic force]\label{prop:2.1}
If $u \in C^3({D})$, and 
\begin{align*}
\sup_{ x\in  D } | u_{xxx}(x) | < \infty,
\end{align*}
then 
\begin{align}
\sup_{x\in D} |-\del PD^\epsilon(u)(x) - (-\del PD^\epsilon_l(u)(x))| &= O(\epsilon),  \label{eq:nonlinear linear force diff}\\
\sup_{x\in D} |-\del PD^\epsilon_l(u)(x) - \bbC u_{xx}(x)| &= O(\epsilon) \label{eq:linear elastic force diff},
\end{align}
so
\begin{align}\label{eq:nonlinear elastic force diff}
\sup_{x\in D} |-\del PD^\epsilon(u)(x) - \bbC u_{xx}(x)| &= O(\epsilon).  
\end{align}
If $u\in C^4({D})$, and
\begin{align}\label{eq:assump on u bd4}
\sup_{x\in D } | u_{xxxx}(x) | < \infty,
\end{align}
then
\begin{align}\label{eq:linear to elasto 1}
\sup_{x\in D} |-\del PD^\epsilon_l(u)(x) - \bbC u_{xx}(x)| = O(\epsilon^2).
\end{align}

\end{proposition}
{\vskip 2mm}
We introduce the usual $H^1(D)$ norm of a function $f$ defined in $D$ by
\begin{align*}
\Vert f\Vert_1=\sqrt{\int_D|f(x)|^2\,dx+\int_D|f_x(x)|^2\,dx}.
\end{align*}
We now state the theorem which shows that $u^\epsilon \to u$ with rate $\epsilon$ in the $H^1(D)$ norm uniformly in time.
\begin{theorem}[Convergence of nonlinear peridynamics to the linear elastic wave equation in the limit that the horizon goes to zero]\label{thm:conv peri elasti}
Let $e^\epsilon := u^\epsilon - u$, where $u^\epsilon$ is the solution of \cref{eq:peri} and $u$ is the solution of \cref{eq:elasto}. Suppose $u^\epsilon(t) \in C^4(D) $, for all $\epsilon >0$ and $t\in [0,T]$. Suppose there exists $C_1>0$, $C_1$ independent of the size of horizon $\epsilon$, such that
\begin{align*}
\sup_{\epsilon >0} \left\{ \sup_{(x,t)\in D \times J} | u^\epsilon_{xxxx}(t,x) | \right\} < C_1 < \infty.
\end{align*}
Then for $\epsilon<\delta$, there is a constant $C_2 > 0$ independent of $\delta,$ $\epsilon$, such that
\begin{align*}
\sup_{t\in [0,T]} \leftcr \int_{D} \rho |\dot{e}^\epsilon(t,x)|^2 dx + \int_{D} \bbC |e^\epsilon_x(t,x)|^2 dx \rightcr &\leq C_2 \epsilon^2,
\end{align*}
so $u^\epsilon \to u$ in the $H^1(D)$ norm at the rate $\epsilon$ uniformly in time $t \in [0,T]$.
\end{theorem}
{\vskip 2mm}

A stronger convergence result holds for the solutions $u^\epsilon_l(t)$ of the family of linearized peridynamic equations. 
\begin{theorem}[Convergence of linearized peridynamics equation to the linear elastic wave equation in the limit that the horizon goes to zero]\label{thm:conv linperi elasti}
Let $e_l^\epsilon := u_l^\epsilon - u$, where $u_l^\epsilon$ is the solution of \cref{eq:peri linear} and $u$ is the solution of \cref{eq:elasto}. Suppose $u_l^\epsilon(t) \in C^4(D) $, for all $\epsilon >0$ and $t\in [0,T]$. Suppose there exists $C_1>0$, $C_1$ independent of the size of horizon $\epsilon$, such that
\begin{align*}
\sup_{\epsilon >0} \left\{ \sup_{(x,t)\in D \times J} | (u_l^\epsilon)_{xxxx}(t,x) | \right\}< C_1 < \infty.
\end{align*}
Then, there is a constant $C_2 > 0$ independent of $\delta$, $\epsilon$, such that
\begin{align*}
\sup_{t\in [0,T]} \leftcr \int_{D} \rho |\dot{e}_l^\epsilon(t,x)|^2 dx + \int_{D} \bbC |(e_l^\epsilon)_x(t,x)|^2 dx \rightcr &\leq C_2 \epsilon^4,
\end{align*}
so $u_l^\epsilon \to u$ in the $H^1(D)$ norm at the rate $\epsilon^2$ uniformly in time $t \in [0,T]$.
\end{theorem}
{\vskip 2mm}

The proofs of \cref{prop:2.1} and \cref{thm:conv peri elasti} and \cref{thm:conv linperi elasti} is given in \cref{s:proof}. We now discuss the finite element approximation of the peridynamic model and show the consistency of the discretization for both piecewise constant and linear interpolation. 

\section{Discrete approximation}\label{s:fe discr}
In this section, we introduce the spatial discretization for the peridynamics evolution. To introduce the ideas we use a linear continuous interpolation over uniform mesh and write the equation of motion of displacement at the mesh points. This type of approximation has been analyzed in \cite{CMPer-Du2} in the 1-d setting and further extended to higher dimensions in \cite{CMPer-Tian3,CMPer-Du4} for a significant class quasi-static problems with linear kernels different than the ones treated  in this investigation. 


  
Let $h$ characterize the mesh size and be given by the distance between grid points. We let $\overline{D}$ and $\overline{\partial D^\epsilon}$ denote the closure of the sets ${D}$ and $\partial D^\epsilon$. To fix ideas we will suppose $\overline{D}$ and $\overline{\partial D^\epsilon}$ contain an integral number of elements of the mesh. Let $ D_h = D\cap h \bbZ$ and $\partial D^{\epsilon,h} = \partial D^\epsilon \cap h \bbZ$, and let $K = \{i\in \bbZ: ih \in \overline{D}\}$ and $K_\epsilon=\{i\in \bbZ: ih \in \overline{\partial D^\epsilon}\}$. Here $K_\epsilon$ corresponds to the list of nodes located inside the closure of the nonlocal boundary $\partial D^\epsilon$. We assume $x_i = ih$. We define the interpolation operator $I_h[\cdot]$, for a given function $g: \overline{D\cup \partial D^\epsilon} \to \bbR$ as follows
\begin{align*}
\lininter{g(y)} &= \sum_{i \in K \cup K_\epsilon} g(x_i) \phi_i(y),
\end{align*}
where $\phi_i(\cdot)$ is the interpolation function associated to the node $i$ and $\{\phi_i\}_{i\in {K\cup K_\epsilon}}$ is a partition of unity, i.e.,
\[
\sum_{i\in {K\cup K_\epsilon}} \phi_i(x) = 1
\]
for all $x\in \overline{D\cup \partial D^\epsilon}$. 
In order to expedite the presentation we assume the diameter of nonlocal interaction $2\epsilon$ is fixed and always contains an integral number of grid points $2m+1$. For this choice $\epsilon=mh$ where $m$ increases as $h$ decreases. When we investigate $m$ convergence we will allow both $\epsilon$ and $h$  to decrease. 

We also consider extensions of discrete sets defined on the nodes $K\cup K_\epsilon$. We write the function $v(t,x_i)$ defined at node $x_i$ as $v_i(t)$ and define the discrete set $\{{v}_i(t) \}_{K\cup K_\epsilon}$.
The function $\hat{u}^\epsilon_h(t)$ is the extension of discrete set $\{\hat{u}^\epsilon_i(t) \}_{K\cup K_\epsilon}$ using the interpolation functions and is defined by
\begin{align*}
\hat{u}^\epsilon_h(t)=E[\{\hat{u}^\epsilon_i(t) \}_{K\cup K_\epsilon}] &= \sum_{i\in K\cup K_\epsilon} \hat{u}^\epsilon_i (t)\phi_i(x), \qquad \forall x\in \overline{D\cup \partial D^\epsilon}.
\end{align*}
We also have the the body force ${b}_h(t)$ given by the extension of discrete set $\{b_i(t) \}_{K\cup K_\epsilon}$ defined by
\begin{align*}
b_h(t)=E[\{ b_i(t) \}_{K\cup K_\epsilon}] &= \sum_{i\in K\cup K_\epsilon} b_i (t)\phi_i(x), \qquad \forall x\in \overline{D\cup \partial D^\epsilon}.
\end{align*}

Let $\hat{u}^\epsilon_h(t)$ be the solution of following equation
\begin{align}\label{eq:fe approx}
\rho\ddot{\hat{u}}^\epsilon_i(t) &= -\del PD^\epsilon(\hat{u}^\epsilon_h(t))(x_i) + b_i(t),
\end{align}
with initial condition defined at the nodes given by
\begin{align*}
\hat{u}^\epsilon_i(0) = f(x_i), \qquad \dot{\hat{u}}^\epsilon_i(0) = g(x_i),  \qquad \forall i \in K ,
\end{align*}
or equivalently given by the extension of the discrete sets
\begin{align}\label{eq: fe approxic}
\hat{u}^\epsilon_h(0) = f_h, \qquad \dot{\hat{u}}^\epsilon_h(0) = g_h,  \qquad \forall i \in K ,
\end{align}
and homogeneous boundary condition given by
\begin{align}\label{eq:fe approxbc}
\hat{u}^\epsilon_i(t) = 0, \qquad \dot{\hat{u}}^\epsilon_i(t)=0, \qquad \forall i \in K_\epsilon.
\end{align}


Similarly, the discrete set $\{\hat{u}^\epsilon_{l,i}(t) \}_{i\in K\cup K_\epsilon}$, with subscript $l$, is extended by interpolation to the function  $\hat{u}^\epsilon_{l,h}(t)=E[\{\hat{u}^\epsilon_{l,i}(t) \}_{i\in K\cup K_\epsilon}]$ and satisfies the linear peridynamic equation 
\begin{align}\label{eq:fe lin approx}
\rho\ddot{\hat{u}}^\epsilon_{l,i}(t) &= -\del PD^\epsilon_{l}(\hat{u}^\epsilon_{l,h}(t))(x_i) + b_i(t) \notag \\
&=\dfrac{2}{\epsilon^2} \sum_{\substack{j \in K\cup K_\epsilon,\\
j\neq i}} f'(0) (\hat{u}^\epsilon_{l,j}(t)-\hat{u}^\epsilon_{l,i}(t))\int_{x_i - \epsilon}^{x_i+\epsilon} \frac{\phi_j(y) J(|y-x_i|/\epsilon)}{|y-x_i|} dy + b_i(t),
\end{align}
with initial conditions \cref{eq: fe approxic} and boundary conditions \cref{eq:fe approxbc}.

We now write \cref{eq:fe lin approx} in vector form and in the next section we will use this representation to provide an explicit stability constraint on time step and mesh size for the the linear peridynamic evolution. Let $U_{l,h}(t) = (\hat{u}^\epsilon_{l,i}(t))_{i\in K} $ be the vector of the approximate solution evaluated at the nodes. Then, \cref{eq:fe lin approx} can be written as
\begin{align}\label{eq:fe lin matrix form}
\rho \ddot{U}_{l,h}(t) &= A U_{l,h}(t) + B(t),
\end{align}
where $a_{ij}$ are defined as
\begin{align}\label{eq: matrixa}
a_{ij} &= \begin{cases}
\bar{a}_{ij} &\qquad \text{if } j\neq i, \\
-\sum_{\substack{k \neq i,\\
k \in K\cup K_\epsilon}} \bar{a}_{ik} &\qquad \text{if } j = i
\end{cases}
\end{align}
where 
\begin{align}\label{eq:def matrix A}
\bar{a}_{ij} &= \dfrac{2}{\epsilon^2} {f'(0)} \int_{x_i - \epsilon}^{x_i + \epsilon} \frac{\phi_j(y) J (|y - x_i|/\epsilon) }{|y-x_i|}dy.
\end{align}
$B(t) = (b_i(t))_{i\in K}$ is the body force vector with 
\begin{align*}
b_i(t) = b(t,x_i ).
\end{align*}

We point out that nonzero nonlocal boundary conditions can be prescribed on $\partial D^\epsilon$. To do this use the standard approach and include the known displacements corresponding to the nonlocal boundary $K_\epsilon$ on the right hand side vector according to the rule, 
\begin{align*}
b_i(t) = b(t,x_i ) + \sum_{j \in K_\epsilon, j \neq i} \bar{a}_{ij} \hat{u}^\epsilon_{l,j}.
\end{align*}

To fix ideas we first use linear continuous interpolation functions $\phi_i(x)$. 



\textbf{Linear continuous interpolation: }Let $i \in K\cup K_\epsilon$. We define $\phi_i(x)$ as follows
\begin{align*}
\phi_i(x) = \begin{cases}
0 & \qquad \text{if } x\notin [x_{i-1}, x_{i+1}], \\
\dfrac{x - x_{i-1}}{h}& \qquad \text{if } x\in [x_{i-1}, x_{i}], \\
\dfrac{x_{i+1} - x}{h}& \qquad \text{if } x\in [x_{i}, x_{i+1}],
\end{cases}
\end{align*}
with $x_{i+1}-x_i=h$, $i\in K\cup K_\epsilon$ and 
\begin{align*}
\sum_{i\in K\cup K_\epsilon}\phi_i=1,
\end{align*}
and for $g\in C^2(\overline{D})$ we have
\begin{align*}
|\lininter{g(x)} - g(x) | &\leq \sup_{z} |g''(z)| h^2.
\end{align*}

\subsection{Consistency error}
\label{section consistency}
We present bounds on the consistency error due to discretization for both the nonlinear peridynamic force and the linearized peridynamic force. The error is seen to depend on  the ratio of mesh size to non-locality, i.e., $h/\epsilon$. The numerical examples in given in \cref{s:numerical} for both linear and nonlinear nonlocal models  corroborate this trend.  

\textbf{$h$-convergence: }We keep $\epsilon$ fixed and estimate the error with respect to mesh size $h$. 

{\vskip 2mm}
\begin{proposition}[Consistency error: peridynamic approximation]\label{prop:3.1}
For linear continuous interpolation, if $u\in C^3(D)$ and $u_{xxx}$ is bounded on $D$  then we have for
linearized peridynamic force
\begin{align}\label{eq:claim 2}
\sup_{i\in K} |\del PD^\epsilon_{l}(\lininter{u})(x_i) -\del PD^\epsilon_l(u)(x_i)| &= O(h/\epsilon),
\end{align} 
and for the nonlinear peridynamic force we have
\begin{align}\label{eq:claim 2nonlin}
\sup_{i\in K} |\del   PD^\epsilon(\lininter{u})(x_i) - \del PD^\epsilon(u)(x_i)| &= O(h/\epsilon).
\end{align} 

\end{proposition}
{\vskip 2mm}

We now examine what happens as $\epsilon$ goes to zero. Combining \cref{prop:2.1}, \cref{prop:3.1} and applying the triangle inequality gives:

{\vskip 2mm}
\begin{proposition}[Consistency error: peridynamic approximation in the limit $\epsilon\to 0$]\label{prop:3.2}
For linear continuous interpolation, if $u\in C^3(D)$ with $u_{xxx}$ bounded then we have for the linearized peridynamic force
\begin{align}
\sup_{i\in K} |-\del  PD^\epsilon_{l}(\lininter{u})(x_i) - \bbC u_{xx}(x_i)| = O(\epsilon) +  O(h/\epsilon)\label{eq:claim 5},
\end{align} 
and for the nonlinear peridynamic force
\begin{align}
\sup_{i\in K} |-\del  PD^\epsilon(\lininter{u})(x_i) - \bbC u_{xx}(x_i)| = O(\epsilon) +  O(h/\epsilon)\label{eq:claim 5nonlinear}.
\end{align} 

\end{proposition}
{\vskip 2mm}
This proposition shows that the consistency error for both nonlinear and linearized nonlocal models are controlled by the ratio of the mesh size to the horizon. This ratio must decrease to zero as the horizon goes to zero in order for the consistency error to go to zero.
We conclude pointing out that the linearized kernels treaded in this work are different than those ones considered in \cite{CMPer-Du2}.

\subsection{Consistency error for higher order interpolation approximation}

It is easy to improve the convergence results if we assume more differentiability for the solution. We  will assume that we have uniform control of  $p+1$ bounded derivatives of solutions with respect to $\epsilon$, and discretize using higher order local Lagrangian shape functions. In this section we estimate the consistency error for this case.   Let $h$ be the mesh size and $p$ be the order of interpolation. The discretization of the domain is now $D_h = D \cap (h/p)\bbZ$ and $\partial D^{\epsilon,h} = \partial D^\epsilon \cap (h/p)\bbZ$. Let $K := \{i\in \bbZ: i(h/p) \in \bar{D} \}$ and $K^\epsilon := \{i\in \bbZ: i(h/p) \in \bar{\partial D^\epsilon}\}$. The mesh points are denoted by $x_i = ih/p$, the interpolation operator is denoted by $\lininter{\cdot}$, and the extension operator is denoted by $E[\cdot]$. The approximate nonlinear peridynamic equation \cref{eq:fe approx}, and approximate linearized peridynamic equation \cref{eq:fe lin approx} are now defined for the $p^{th}$  order interpolations $\{\phi_i\}$. We now state the following results:

{\vskip 2mm}
\begin{proposition}[Consistency error: peridynamic approximation]\label{prop:3.3}
For continuous interpolation of order $p$, if $u\in C^{p+1}(D)$ and the $(p+1)^{\text{th}}$ derivative of $u$ is bounded on $D$  then we have for the linearized peridynamic force
\begin{align}\label{eq:claim 6}
\sup_{i\in K} |\del PD^\epsilon_{l}(\lininter{u})(x_i) -\del PD^\epsilon_l(u)(x_i)| &= O(h^p/\epsilon),
\end{align} 
and for the nonlinear peridynamic force
\begin{align}\label{eq:claim 6nonlin}
\sup_{i\in K} |\del   PD^\epsilon(\lininter{u})(x_i) - \del PD^\epsilon(u)(x_i)| &= O(h^p/\epsilon).
\end{align} 
\end{proposition}
{\vskip 2mm}

Next we examine what happens as we send $\epsilon$ to zero. Combining \cref{prop:2.1}, \cref{prop:3.3} and applying the triangle inequality gives:

{\vskip 2mm}
\begin{proposition}[Consistency error: peridynamic approximation in the limit $\epsilon\to 0$]\label{prop:3.4}
For continuous interpolation of order $p$, if $u\in C^{p+1}(D)$ with $(p+1)^{\text{th}}$ derivative of $u$  bounded then we have for the linearized peridynamic force
\begin{align}
\sup_{i\in K} |-\del  PD^\epsilon_{l}(\lininter{u})(x_i) - \bbC u_{xx}(x_i)| = O(\epsilon) +  O(h^p/\epsilon)\label{eq:claim 7},
\end{align} 
and for the nonlinear peridynamic force
\begin{align}
\sup_{i\in K} |-\del  PD^\epsilon(\lininter{u})(x_i) - \bbC u_{xx}(x_i)| = O(\epsilon) +  O(h^p/\epsilon)\label{eq:claim 7nonlinear}.
\end{align} 

Let $\bar{p} = \max \{p+1, 4\}$. In case of linear peridynamics and $u\in C^{\bar{p}}(D)$ such that $\bar{p}^{\text{th}}$ derivative of $u$ is bounded, we have
\begin{align}
\sup_{i\in K} |-\del  PD^\epsilon_{l}(\lininter{u})(x_i) - \bbC u_{xx}(x_i)| = O(\epsilon^2) +  O(h^p/\epsilon)\label{eq:claim 8}.
\end{align} 
\end{proposition}
{\vskip 2mm}

For $p=1$, we need $u \in C^3(D)$ (see \cref{prop:3.2}). The oultines of proofs are provided in \cref{s:proof}.

\section{The central difference scheme and stability analysis}\label{s:time discr}
In this section, we consider the central difference time discretization of the semi-discrete peridynamic equation \cref{eq:fe approx}. 
We recover a new the stability condition for the linearized peridynamic equation, see \cref{eq: almost CFL}. An explicit stability condition relating $\Delta t$ to $h$  is obtained in terms of the linearized peridynamic material parameters.  It is similar to the standard CFL condition for central difference approximation of 1-d wave equation. 
We point out that the stability of the linearized peridynamic solution can imply the stability of nonlinear peridynamic solution. This implication is physically reasonable provided that the acceleration and body force are sufficiently small and so that one can approximate nonlinear peridynamics by its linearization. 

Let $\Delta t$ be the time step and the field $u(t)$ at time step $k\Delta t $ is denoted by $u^k$. To illustrate ideas we will assume $\rho =1$. For the linearized peridynamics we characterize the matrix $A$ associated with the spatial discretization \cref{eq:fe lin matrix form}. We introduce a special class of matrices.

\noindent {\bf Definition.} An M-matrix has negative off diagonal elements $m_{ij}$, $i\not=j$, and the diagonal elements satisfy $m_{ii} \geq \sum_{j\neq i} m_{ij}$ for all $i$. 

The stability of the numerical scheme is based on the following property of $A$. 
{\vskip 2mm}
\begin{lemma}[Properties of the $A$ matrix]\label{lem:4.1}
For linear interpolations, the square matrix $-A$ of size $|K| \times |K|$ is a Stieltjes matrix, i.e. it is a nonsingular symmetric M-matrix. Therefore, the eigenvalues of $-A$ of is real and positive. 
\end{lemma}
{\vskip 2mm}
\begin{proof}
$-A$ is clearly M-matrix as its off-diagonal terms are negative, and diagonal terms satisfy $-a_{ii} \geq \sum_{j\neq i} -a_{ij}$ for all $i$. To prove that a M matrix is nonsingular, we apply Theorem 2.3 in [\cite{MALa-Berman}, Chapter 6]. 
From the definition of $-A$ we find that
\begin{align*}
-a_{ii}=\sum_{i\not=j}{-a_{ij}},\,\,\,\,i=1,\ldots,n,,\qquad -a_{ii}>\sum_{j=1}^{i-1}{-a_{ij}}, \,\,\,\,i=2,\ldots,n,
\end{align*}
and this is easily seen to be condition $M_{37}$ of Theorem 2.3 and we conclude that $-A$ is nonsingular.
The symmetry  of $-A$ is a straight forward consequence of the formula \cref{eq:def matrix A}. 
\end{proof}

\textbf{Central difference time discretization: }For $\rho=1$, the spatially discretized evolution equations for linearized peridynamics given by \cref{eq:fe lin matrix form} is written
\begin{align*}
\ddot{U}_{l,h}(t) &= A U_{l,h}(t) + B(t).
\end{align*}
We now additionally discretize in time using the central difference scheme. Let $U^k_{l,h} := \{\hat{u}^{\epsilon,k}_{l,i}\}_{i\in K}$ denote the discrete displacement field at time step $k$. Here we use the subscript ``$l$'' for linear peridynamic and superscript ``$\epsilon$'' to highlight that the solution corresponds to size of horizon $\epsilon$. In what follows, we will assume no body force and the dynamics is driven by the initial conditions. 
Since we have the zero Dirichlet boundary condition, we know the displacement at nodes $i \in K_\epsilon$ is zero for all time steps. We assume  $k \leq T /\Delta t$, and  the horizon is given by $\epsilon=mh/2$ where $m$ is a positive integer.
The discretized dynamics is given by the solution $\{U^k_{l,h}\}$ of the following equation
\begin{align*}
\frac{U^{k+1}_{l,h}-2 U^k_{l,h} + U^{k-1}_{l,h}}{\Delta t^2} = A U^k_{l,h},
\end{align*}
or after elementary manipulation
\begin{align}\label{eq:central scheme}
U^{k+1}_{l,h} = - U^{k-1}_{l,h} + (2+\Delta t^2 A) U^k_{l,h}.
\end{align}

\begin{theorem}[Stability criterion for the central difference scheme]\label{thm:4.2}
Recall the elastic constant $\mathbb{C}$ given by \cref{eq:elasto const C}, $f'(0)$ given by \cref{eq:per asymptote}, and  $M=\max_{0<r\leq 1}\{J(r)\}$. Then the central difference scheme \cref{eq:central scheme}, in the absence of body forces, is stable as long as $\Delta t$ satisfies 
\begin{align}\label{eq: almost CFL}
\Delta t \leq \frac{h}{\sqrt{\mathbb{C}+ 2f'(0)\frac{Mh^2}{\epsilon^2}}}.
\end{align}

\end{theorem}

{\vskip 2mm}

\noindent {\bf Remark.} The stability condition for the linear elastic wave equation is given by the CFL condition $\Delta t\leq\frac{h}{\sqrt{\mathbb{C}}}$ where $h$ gives the distance between mesh points. 

\begin{proof}
Let $(\gamma_i,\bnu_i)$ be an eigenpair of $A$. Let $\lambda_i = -\gamma_i$, then $\lambda_i >0$ and let $\lambda= \max_i\{\lambda_i\}$. Substitute $U^k_{l,h} = \xi^k \bnu$, where $\xi$ is some real number and by $\xi^k$ we mean the $k^{\text{th}}$ power of $\xi$, into \cref{eq:central scheme}, to obtain the characteristic equation 
\begin{align*}
\xi^{2} - 2 \theta \xi + 1 = 0.
\end{align*}
where $\theta = 1 - 1/2 \lambda_i \Delta t^2$. The solution of the quadratic equation gives two roots: $\delta_1 = \theta + \sqrt{\theta^2-1}$ and $\delta_2 = \theta - \sqrt{\theta^2 - 1}$. We need $\abs{\delta} \leq 1$ for stability. Since $\delta_1 \delta_2 = 1$, the only possibility is when $|\delta_1|=|\delta_2|=1$. This is satisfied for all eigenmodes when 
\begin{align*}
&|\theta| \leq 1 \notag \\
\Rightarrow & \Delta t \leq \dfrac{2}{\sqrt{\lambda}}\leq \dfrac{2}{\sqrt{\lambda_i}}.
\end{align*}
A lower estimate on $1/\sqrt{\lambda}$ follows from Gershgorin's circle theorem:
\begin{theorem}
Any eigenvalue of $A$ lies inside at least one of the disks
\begin{equation}
|\gamma-a_{ii}|<\sum_{i\not= j}|a_{ij}|.
\label{G}
\end{equation}
\end{theorem}
All eigenvalues of $A$ lie on the negative real axis and we provide an upper estimate on the largest magnitude of the eigenvalues depending only on the mesh size $h$ given by the distance between interpolation points and the horizon $\epsilon=mh$. For this case, it follows from Equation \cref{G} and \cref{eq: matrixa} that
\begin{equation*}
\lambda<2\sum_{i\not= j}\overline{a}_{ij},
\end{equation*}
Writing out the sum and using the definition of the interpolating functions and their partition of unity properties we get
\begin{align*}
\sum_{i\not=j}\overline{a}_{ij}=\frac{2f'(0)}{\epsilon^2}\int_{x_{i-1}}^{x_{i+1}}\frac{1}{h}J(|y-x_i|/\epsilon)\,dy \notag \\
+\frac{2f'(0)}{\epsilon^2}\int_{x_{i}-\epsilon}^{x_{i-1}}\frac{J(|y-x_i|/\epsilon)}{|y-x_i|}\,dy \notag \\
+\frac{2f'(0)}{\epsilon^2}\int_{x_{i+1}}^{x_{i}+\epsilon}\frac{J(|y-x_i|/\epsilon)}{|y-x_i|}\,dy.
\end{align*}
Here we make use of the identities
\begin{align*}
1&=\sum_{j\in I^+}\phi^j(y), y\in[x_{i+1},x_i+\epsilon],\\
1&=\sum_{j\in I^-}\phi^j(y), y\in[x_{i}-\epsilon,x_{i-1}],
\end{align*}
where $I^+=\{j\,:x_j\in[x_{i+1},x_i+\epsilon]\}$ and $I^-=\{j\,:x_j\in[x_{i}-\epsilon,x_{i-1}]\}$ .
For $y<x_{i-1}$ and $x_{i+1}<y$ we have $h<|y-x_i|$ and $1<|y-x_i|/h$ and we have the estimate
\begin{align*}
\sum_{i\not=j}\overline{a}_{ij} &\leq 2\frac{\mathbb{C}}{h^2}+\frac{2f'(0)}{h\epsilon^2}\int_{x_{i-1}}^{x_{i+1}}J(|y-x_i|/\epsilon)\,dy \notag \\
&\leq 2\frac{\mathbb{C}}{h^2}+\frac{4}{\epsilon^2}f'(0)M,
\end{align*}
and a lower bound now follows on $1/\sqrt{\lambda}$. Simple manipulation then delivers \cref{eq: almost CFL}. 
\end{proof}
{\vskip 2mm}

\section{Numerical simulation}\label{s:numerical}
In this section we present numerical simulations that independently corroborate the theoretical bounds on the consestency error given in section \cref{section consistency}. We start in  \cref{ss:nondim} and pose the non-dimensional initial boundary value problem. We then perform a numerical study of the $h$-convergence in \cref{ss: conoverge} and convergence with respect to the ratio $h/\epsilon$ in \cref{ss: mconverge}.  We compare the numerical simulations for the nonlinear and linear nonlocal models with local linear elastodynamics.

\subsection{Nondimensional peridynamic equation}\label{ss:nondim}
Let $[0,L]$ be the bar with length $L$ in meters. Let $[0,T]$ be the time domain in units of seconds. Given a dimensionless influence function $J(r), r\in [0,1]$, the bond force $f'(0)$ is in the units $N/m^2$, and the density $\rho$ in unit $kg/m^3$, the wave velocity in an equivalent  linear elastic medium can be determined by
\begin{align*}
\nu_0 = \sqrt{f'(0) M/\rho}, \quad M := 2 \int_0^1J(r) r dr .
\end{align*}
We introduce the time scale $T_0:= L/\nu_0$. Then a wave in the elastic media with elastic constant $\bbC= Mf'(0)$ requires $T_0$ seconds to reach from one end of the bar to the other end.

We let $\bar{x} = x/L$ for $x\in [0,L]$, and $\bar{t} = t/T_0$. We define non-dimensional solution $\bar{u}(\bar{x}, \bar{t}) := u(L\bar{x}, T_0 \bar{t})/L$. Let $\bar{\epsilon} := \epsilon/L$ be nondimensional size of horizon. Then $\bar{u}$ satisfies
\begin{align*}
\ddot{\bar{u}}(\bar{t}, \bar{x}) = \dfrac{2}{\bar{\epsilon}^2} \int_{\bar{x} - \bar{\epsilon}}^{\bar{x} + \bar{\epsilon}} \bar{f}'(|\bar{y} - \bar{x}| \bar{S}^2) \bar{S}(\bar{y}, \bar{x}) J(|\bar{y} - \bar{x}|/\bar{\epsilon}) d\bar{y} + \bar{b}(\bar{t}, \bar{x}),
\end{align*} 
where $\bar{S}(\bar{y},\bar{x}; \bar{u}(\bar{t})) = (\bar{u}(\bar{t}, \bar{y}) - \bar{u}(\bar{t}, \bar{x}))/|\bar{y} - \bar{x}|$, $\bar{f}'(r) = \dfrac{L}{\bbC} f'(Lr)$, and $\bar{b}(\bar{t}, \bar{x}) = \dfrac{L}{\bbC} b(T_0\bar{t}, L \bar{x})$.
The time interval $T_0$ for a given $E=f'(0)$ is given by $T_0 = L \sqrt{\rho/EM}$ and $u(t) = L \bar{u}(t/T_0)$.

In the following studies we choose the influence function to be $J(|x|) = 2|x|\exp(-|x|^2/\alpha)$ with $\alpha= 0.4$. 
The nonlinear potential function $f$ is taken to be  $f(|x|{S}^2) = C (1 - \exp[- b |x|{S}^2])$. We let $b = 1$ and $f'(0)= Cb = C = 1/M$, where $M= 2\int_0^1 J(r)dr$. This gives $T_0 = 1$. The body force is set to zero, i.e. $b=0$. All numerical results shown in this article will correspond to above choice of $J$, $b$, and $f$.


\subsection{$h$-convergence} \label{ss: conoverge}
We study the the rate of convergence as seen in the simulations for two different choices of initial conditions.
In first problem, we consider the Gaussian pulse as the initial condition given by: $u_0(x) = a \exp[-(0.5-x)^2/\beta], v_0(x)= 0.0$ with $a= 0.005$ and $\beta= 0.00001$. The time interval is $[0,1.7]$ and the time step is $\Delta t= 0.00001$. We fix $\epsilon$ to $0.1$, and consider the mesh sizes $h = \{\epsilon/10, \epsilon/100, \epsilon/1000\}$. 
For the second problem, we consider the double Gaussian curve as initial condition: $u_0(x) = a \exp[-(0.25-x)^2/\beta] + a \exp[-(0.75-x)^2/\beta], v_0(x)= 0.0$ with $a= 0.005$ and $\beta= 0.00001$. The time interval for the second problem is $[0,0.5]$ and the time step is $\Delta t= 0.000005$. Here we consider a smaller horizon $\epsilon= 0.01$, and solve for the three mesh sizes $h =\{\epsilon/100, \epsilon/200, \epsilon/400 \}$.

Using the approximate solutions corresponding to three different mesh sizes, we can easily compute the dependence of the error with respect to mesh size $h$. Let $u_1,u_2,u_3$ correspond to meshes of size  $h_1,h_2,h_3$, and let $u$ be the exact solution. We write the error as $||u_h - u|| =C h^\alpha$ for some constant $C$ and $\alpha>0$, and fix the ratio of mesh size  $h_1/h_2 = h_2/h_3 = r$, to get
\begin{align*}
\log( ||u_1 - u_2||) &= C + \alpha \log h_2, \\
\log( ||u_2 - u_3||) &= C + \alpha \log h_3.
\end{align*}
Then the rate of convergence $\alpha$ is $$\dfrac{\log( ||u_1 - u_2||) - \log( ||u_2 - u_3||)}{\log(r)}.$$

In \cref{table:result problem 1} and \cref{table:result problem 2}, we list lower bound on the rate of convergence for different times in the evolution. The rate of convergence for the simulation is seen to depend on the time. We also note that the rate of convergence for the linear peridynamic solution is very close to that of the nonlinear peridynamic solution and both convergence rates lie above the theoretically predicted convergence rate for the $L^2$ error given by $\alpha=1$. 

\begin{table} 
\caption{Convergence result for problem 1. Superscript 1 refers to $L^2$ norm and 2 refers to $\sup$ norm. NPD refers to nonlinear peridynamic and LPD refers to linear peridynamic. Max time step is $170000$.}
\centering
\begin{tabular}{lllll} \toprule
Time step & $\text{LPD}^1$ & $\text{NPD}^1$ & $\text{LPD}^2$ & $\text{NPD}^2$ \\ \midrule 
6000  & 1.6416 & 1.6419 & 1.4204 & 1.4204 \\
51500   & 1.3098 & 1.3106 & 1.3312 & 1.3331 \\
104000   & 1.1504 & 1.1482 & 1.5155 & 1.5557 \\
147000   & 1.1364 & 1.1262 & 1.6027 & 1.5215 \\
165000   & 1.2611 & 1.2632 & 1.5496 & 1.6055 \\
\bottomrule
\end{tabular}
\label{table:result problem 1}
\end{table}

\begin{table} 
\caption{Convergence result for problem 2. Superscript 1 refers to $L^2$ norm and 2 refers to $\sup$ norm. NPD refers to nonlinear peridynamic and LPD refers to linear peridynamic. Max time step is $10^5$.}
\centering
\begin{tabular}{lllll} \toprule
Time step & $\text{LPD}^1$ & $\text{NPD}^1$ & $\text{LPD}^2$ & $\text{NPD}^2$ \\ \midrule 
2000   & 1.4498 & 1.4504 & 1.2546 & 1.2547 \\
54000   & 1.3718 & 1.3707 & 1.6903 & 1.6908 \\
96000   & 1.3735 & 1.3719 & 1.3753 & 1.3816 \\
\bottomrule
\end{tabular}
\label{table:result problem 2}
\end{table}

\subsection{Convergence with respect to $h$ and $h/\epsilon$}\label{ss: mconverge}
We consider the limit of the peridynamic solution as $\epsilon \to 0$. The initial displacement is  $u_0(x) = a \exp[-(0.5-x)^2/\beta], v_0(x)= 0.0$ with $a= 0.005$ and $\beta= 0.00001$. The time domain is taken to be $[0,0.1]$ and the time step is $\Delta t= 0.0000005$. We fix the ratio $\epsilon/h = 100$, and solve the problem for three different peridynamic horizons given by $\epsilon= 0.0016$, $\epsilon=0.0008$, and $\epsilon=0.0004$. As before we assume a convergence error $\Vert u^\epsilon-u\Vert\leq C \epsilon^\alpha$. The rate is of convergence in the simulations is measured by
$$\dfrac{\log(||u^{\epsilon_1} - u^{\epsilon_2}||) - \log(||u^{\epsilon_2} - u^{\epsilon_3}||)}{\log(\epsilon_2) - \log(\epsilon_3)}.$$
In \cref{table:eps converg test 20}, we record the convergence rate with respect to $\epsilon$  for different times in the evolution.

\begin{table}
\caption{Rate of convergence $\frac{\log(||u^{\epsilon_1} - u^{\epsilon_2}||) - \log(||u^{\epsilon_2} - u^{\epsilon_3}||)}{\log(\epsilon_2) - \log(\epsilon_3)}$.}
\centering
{\begin{tabular}{@{}ccccc@{}} \toprule
\multicolumn{1}{c}{Time step} & \multicolumn{2}{c}{Conv. of LPD} & \multicolumn{2}{c}{Conv. of NPD} \\
& $L^2$ & $\sup$ & $L^2$ & $\sup$\\ \midrule 
  \cmidrule{1-1}\cmidrule{2-3}\cmidrule{4-5}
2000   & 1.9052 & 1.5556 & 1.9052 & 1.5556 \\
50000  & 1.7916 & 1.6275 & 1.7916 & 1.6275 \\
100000 & 1.718  & 1.449  & 1.718  & 1.449 \\
150000 & 1.6298 & 1.2688 & 1.6298 & 1.2688 \\
200000 & 1.5388 & 1.1086 & 1.5388 & 1.1086 \\
\bottomrule
\end{tabular}}
\label{table:eps converg test 20}
\end{table}

\begin{figure}[ht]
\centering
\includegraphics[scale=0.4]{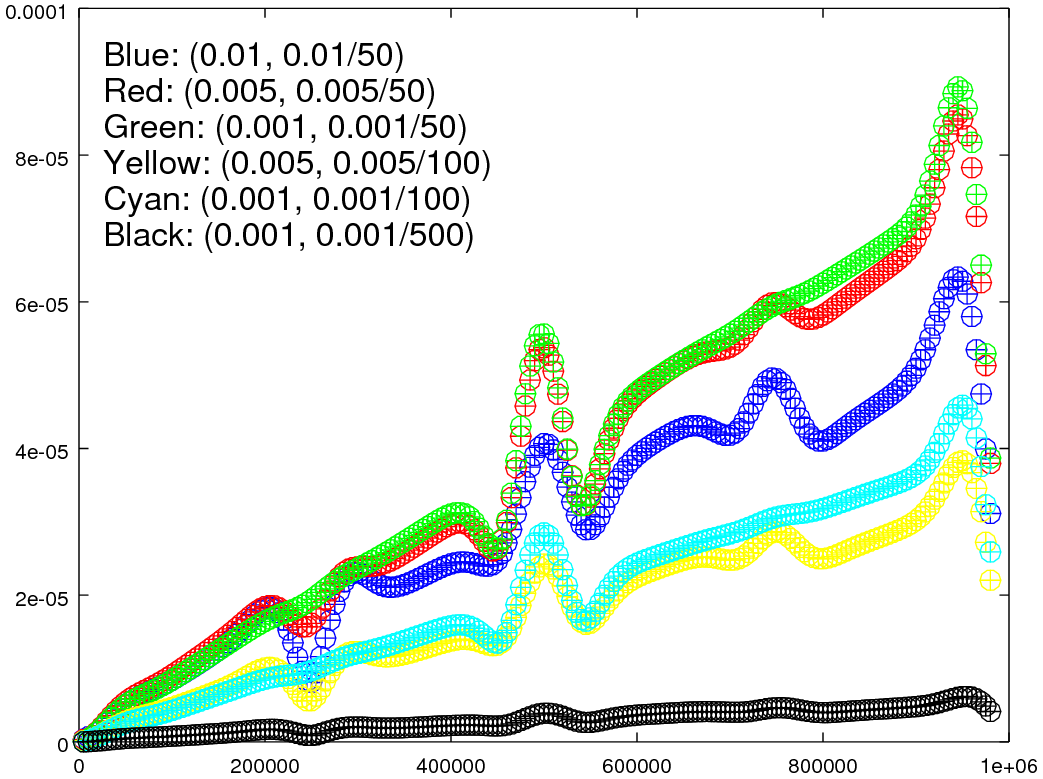}
\caption{Plot of $||u_{peri} - u_{elasto}||_{L^2}$ at different time steps. Arguments inside the bracket corresponds to $(\epsilon, h)$. ``+'' corresponds to the linear peridynamics and ``o'' corresponds to nonlinear peridynamics. For $(\epsilon= 0.005, h= \epsilon/100)$ (Yellow curve), the error $||u_{peri} - u_{elasto}||$ is smaller compared to the error for $(\epsilon= 0.01, h=\epsilon/50)$ (Blue curve), whereas for the same $\epsilon=0.005$ but with $h= \epsilon/50$ (Red curve), error is in fact higher than the error corresponding to $(\epsilon= 0.01, h=\epsilon/50)$ (Blue curve). To further demonstrate the dependence of $||u_{peri} - u_{elasto}||$ on $h/\epsilon$, the solution corresponding to $(\epsilon= 0.001, h =\epsilon/100)$ (Cyan curve) lies above the Yellow curve. However, when the ratio $\epsilon/h$ is increased to $500$ (Back curve), i.e. for $(\epsilon= 0.001, h= \epsilon/500)$, we see that the Black curve is lower than the Yellow curve. Also note that the error plot corresponding to linear and nonlinear peridynamics are almost same (``+'' and ``o'' overlap in each curve).}\label{fig:error plot m converg}
\end{figure}

\textbf{Comparison with the elastodynamic solution: }Next we compare the numerical solutions of elastodynamics, linear peridynamics, and nonlinear peridynamics. The comparison is made using the common initial data: $u_0(x) = a \exp[-(0.25-x)^2/\beta] + a \exp[-(0.75-x)^2/\beta], v_0(x)= 0.0$ with $a= 0.001$ and $\beta= 0.003$. The time interval for simulation is $[0,1.0]$ and the time step is $\Delta t=  0.000001$. The time interval has been chosen sufficiently large to include the effect of wave reflection off the  boundary. In \cref{fig:error plot m converg}, we plot the error $||u_{peri} - u_{elasto}||$ at each time step. \cref{fig:error plot m converg} validates the fact that error depends on $h/\epsilon$ (see \cref{eq:claim 5nonlinear} and \cref{eq:claim 5}). In \cref{fig:compare solutions}, we plot the solutions at different time steps. 

In \cref{fig:error plot m converg}, we see that error has a jump when $t$ is close to $0.25, 0.5, 0.75, 0.95$. The jump near $t=0.25$ and $t=0.75$ is due to the wave dispersion effect when the wave hits the boundary. The reason for this is that for peridynamic simulations with smaller $\epsilon$ (compare Green, Cyan, and Black curve in \cref{fig:error plot m converg} with that of large $\epsilon$ in Blue, Red, and Yellow curve), the jump in error near $t=0.25$ and $t=0.75$ goes away irrespective of the $h/\epsilon$ ratio. As for the jump in error near $t=0.5$ and $t=0.95$, we look at the simulation and find that close to time $t=0.5,0.95$, there is interaction an between two Gaussian pulses traveling towards each other. This interaction is well captured by peridynamic solution when $\epsilon$ is small along with a small ratio  $h/\epsilon$. The Cyan curve corresponds to smaller $\epsilon$ as compared to the Blue curve. But the jump near $t=0.5$ and $t=0.95$ does not improve much in Cyan curve. However, when we consider the finer mesh used in the simulation corresponding to the Black curve with $\epsilon$ same as that of Cyan curve, the jump is greatly reduced. 

\begin{figure}[ht]
\centering
\subfloat[$t= 0.25$.]{\includegraphics[width=0.3\textwidth]{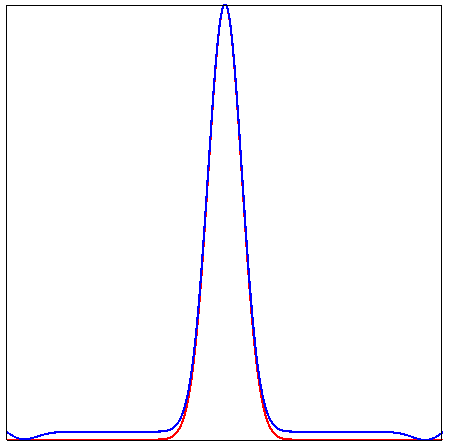}} \hspace{1cm}
\subfloat[$t= 0.505$.]{\includegraphics[width=0.3\textwidth]{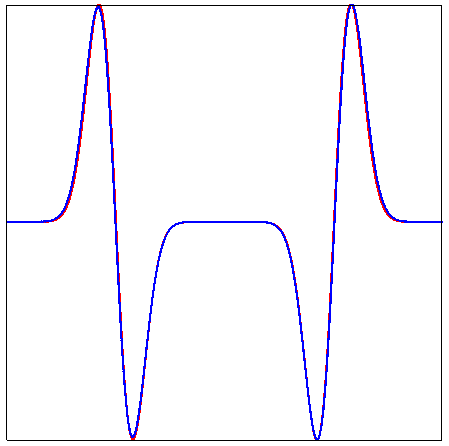}}\\
\subfloat[$t= 0.75$.]{\includegraphics[width=0.3\textwidth]{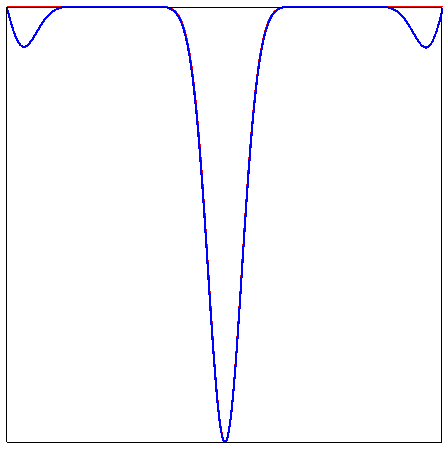}}\hspace{1cm}
\subfloat[$t= 1.0$.]{\includegraphics[width=0.3\textwidth]{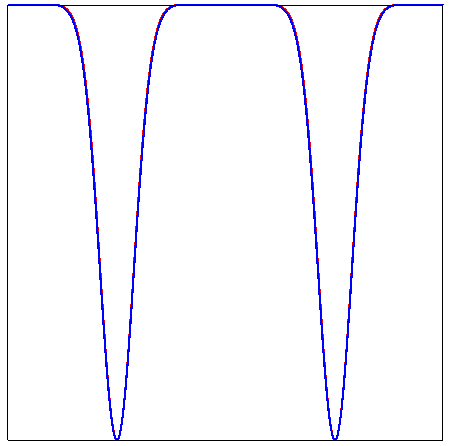}}

\caption{The elastodynamic solution is shown in Red, linear peridynamics in Green, and nonlinear peridynamics in Blue. Simulation shows that solutions of linear and nonlinear peridynamics are nearly identical. The Green curve is hidden beneath Blue curve. The elastodynamic solution corresponds to mesh size $h=0.00001$ whereas the peridynamic solution corresponds to $\epsilon=0.005$ and $h= \epsilon/100$. Plots above are normalized so that the displacement lies within $[0,1]$.}
  \label{fig:compare solutions}
\end{figure}

The difference between the red and blue curves in \cref{fig:compare solutions} at $t=0.25$ and $t=0.75$ is due to the presence of wave dispersion in the nonlocal model and reflection of the pulses by the boundary as described in \cref{fig:error plot m converg}. The difference in red and blue curves at $t=0.5$ and $t=1.0$ is due to the interaction between the pulses as they approach each other and associated approximation  error for the nonlocal model described in \cref{fig:error plot m converg}.

%
%
\textbf{Comparison between nonlinear and linear peridynamic solutions: }In \cref{prop:2.1}, we have shown that difference between the nonlinear and linearized peridynamic force is controlled by $\epsilon$ when solution is smooth. Therefore, we would expect that as the size of horizon gets smaller the difference between approximate solution of linear and nonlinear peridynamics will get smaller. Let $u^1_l, u^2_l$ be the linear peridynamic solution and $u^1, u^2$ be the nonlinear peridynamic solution. ``1'' corresponds to $(\epsilon_1= 0.01, h_1= \epsilon/50)$ and ``2'' $(\epsilon_2= 0.005, h_2= \epsilon/100)$. \cref{fig:LPD NPD comparison} shows the plot of slope $\frac{\log(||u^1 - u^1_l||_{L^2}) - \log(||u^2 - u^2_l||_{L^2})}{\log(\epsilon_1) - \log(\epsilon_2)}$ at different time steps. We see from the figure that the rate of convergence is very consistent with respect to time and is very close to expected value $1$. 

\begin{figure}[ht]
\centering
\includegraphics[scale=0.35]{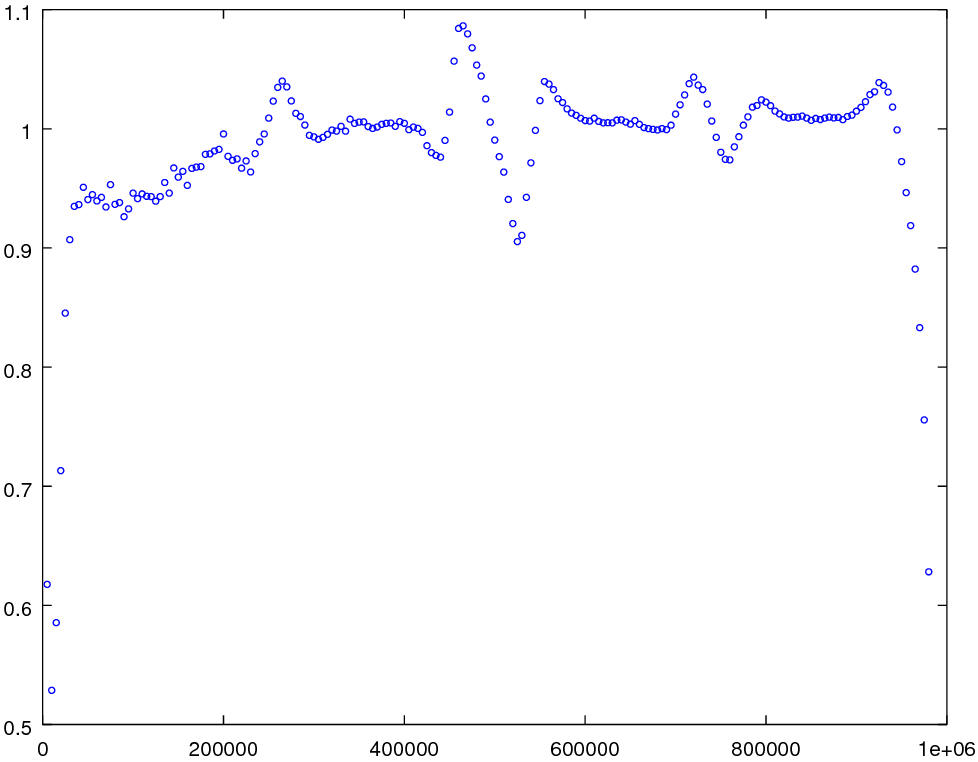}
\caption{Slope of $\log(||u_{NPD} - u_{LPD}||_{L^2})$ with respect to $\log(\epsilon)$ at different time steps from $k=0$ to $k=10^6$.}\label{fig:LPD NPD comparison}
\end{figure}

\section{Convergence of nonlinear nonlocal models to local elastodynamics in dimensions $2$ and $3$}\label{ss: higherD}

We display the convergence of the nonlinear nonlocal model to elastodynamics in dimensions $2$ and $3$. In general for $d=1,2,3$, the nonlinear nonlocal force is given by
\begin{align*}
-\del PD^\epsilon(u)(x) &= \dfrac{4}{\epsilon^{d+1} \omega_d} \int_{H_\epsilon(x)} J(|y-x|/\epsilon) f'(|y-x| S(y,x;u)^2) S(y,x;u)e_{y-x}dy,
\end{align*}
where $u\in L^2(D;\bbR^d)$, $H_\epsilon(x)$ ball of radius $\epsilon$ centered at $x$ in $\bbR^d$, $\omega_d$ is volume of unit ball in $\bbR^d$, $e_{y-x}=\frac{y-x}{|y-x|}$, and $J$, $f$ are the same as before. 

{\vskip 2mm}
\begin{proposition}[Control on the difference between peridynamic force and elastic force]\label{prop:6.1}
Let $D$ be a bounded domain in $\bbR^d$. If $u \in C^3({D}; \bbR^d)$, and $\sup_{x\in D } | \del^3 u(x) | < \infty$ then 
\begin{align*}
\sup_{x\in D} |-\del PD^\epsilon(u)(x) - \del \cdot \bar{\bbC} \mathcal{E} u(x)| &= O(\epsilon),
\end{align*}
where $\bar{\bbC}$ is given by
\begin{align}\label{eq:elastic tensor dd}
\bar{\bbC} = \dfrac{2f'(0)}{\omega_d} \int_{H_1(0)} J(|\xi|) e_\xi \dyad e_\xi \dyad e_\xi \dyad e_\xi |\xi| d\xi,
\end{align}
$e_\xi = \xi/|\xi|$ and the strain tensor is $\mathcal{E}u(x)=(\del u(x)+\del u^T(x))/2$.

\end{proposition}
{\vskip 2mm}
In this treatment we define the boundary $\partial D$ of $D\subset\mathbb{R}^d$ in the usual way as the set of limit points of $D$. Similar to the case of one-dimension, we consider $u=0$ on $\partial D$ and extend $u$ by zero by zero outside $D$. We prescribe a nonlocal boundary condition on $u^\epsilon$ given by $u^\epsilon=0$ on $\{x\in \bbR^d: dist(x, \partial D) \leq \epsilon \}$. The  initial conditions for $u$ and $u^\epsilon$ are the same and given by $u_0$ and $v_0$ on $D$ with $u_0$ and $v_0$ defined on $\mathbb{R}^d$, $d=2,3$ vanishing outside $D'\subset D$ such that $dist(\partial D',\partial D)>0$. We have

\begin{theorem}[Convergence of nonlinear peridynamics to the linear elastic wave equation in the limit that the horizon goes to zero]\label{thm:6.2}
Let $e^\epsilon := u^\epsilon - u$, where $u^\epsilon$ is the solution of peridynamics equation 
\begin{align}
\rho\ddot{u}^\epsilon(t,x) &= -\del PD^\epsilon(u^\epsilon(t))(x) + b(t,x), \label{eq:perivec} 
\end{align}
and $u$ is the solution of elastodynamics equation
\begin{align}
\rho\ddot{u}(t,x)=\del \cdot \bar{\bbC} \mathcal{E}u(t,x)+b(t,x), \label{eq:elastovec} 
\end{align}
with elastic tensor given by \cref{eq:elastic tensor dd}. We assume that $u^\epsilon$ and $u$ satisfy same initial condition, and $u=0$ on $\partial D$. Suppose $u^\epsilon(t) \in C^4(D; \bbR^d) $, for all $\epsilon >0$ and $t\in [0,T]$. Suppose there exists $C_1>0$, $C_1$ independent of the size of horizon $\epsilon$, such that
\begin{align*}
\sup_{\epsilon >0} \left[ \sup_{(x,t)\in D \times J} | \del^4 u^\epsilon(t,x) | \right] < C_1 < \infty.
\end{align*}
Then for $\epsilon$ such that $dist(\partial D',\partial D)>\epsilon>0$, there $\exists C_2 > 0$ such that
\begin{align*}
\sup_{t\in [0,T]} \leftcr \int_{D} \rho |\dot{e}^\epsilon(t,x)|^2 dx + \int_{D} \mathcal{E} e^\epsilon(t,x) \cdot \bar{\bbC} \mathcal{E} e^\epsilon(t,x) dx \rightcr &\leq C_2 \epsilon^2
\end{align*}
so $u^\epsilon \to u$ in the $H^1(D; \bbR^d)$ norm at the rate $\epsilon$ uniformly in time $t \in [0,T]$.
\end{theorem}
{\vskip 2mm}

The proof is similar to the case of one dimension except in this case vector nature of displacement field has to be considered. Following the steps in \cref{s:proof}, \cref{prop:6.1} and \cref{thm:6.2} can be shown and therefore we omit the proof.

\section{Proof of claims}\label{s:proof}
In this section, we will present the proof of claims in \cref{s:intro} and \cref{s:fe discr}.  For simplification, we adopt the following notation
\begin{align}\label{eq:notat 1}
p &:= u_x(x), \qquad q := u_{xx}(x), \qquad r := u_{xxx}(x), \notag \\
e &:= \frac{(y-x)}{|y-x|}=sign\{y-x\}.
\end{align}
In  proving results related to consistency error, we will employ the Taylor series expansion of $u(y)$ with respect to point $x_i$. Since the potential $f$ is assumed to be sufficiently smooth, $f''(r), f'''(r),$ and $f^{''''}(r)$ are bounded for $0<r<\infty$.

\subsection{Bound on difference of peridynamic, linear peridynamic, and elastodynamic force}
We prove \cref{prop:2.1} for $u \in C^3({D})$. Using Taylor series expansion, we get
\begin{align*}
S(y,x; u) &=u_x(x) \dfrac{y-x}{|y-x|} + 1/2u_{xx}(x) |y-x| + 1/6 u_{xxx}(\xi) |y-x| (y-x) \notag \\
&= pe + q|y-x|/2 + T_1(y - x)/|y-x| ,
\end{align*}
where $T_1 = O(|y-x|^3)$. On taking the Taylor series expansion of the nonlinear potential, and substituting in the expansion above, we get
\begin{align*}
&\left(f'(|y-x|S(y,x;u)^2) - f'(0) \right) S(y,x;u) \notag \\
&= f''(0) p^3 |y-x| e + (f''(0) p^2 q3/2 + f'''(0)p^5e/2) |y-x|^2 + T_2(y-x),
\end{align*}
where $T_2(y-x) = O(|y-x|^3)$. Using the previous equation, we get
\begin{align*}
& -\del PD^\epsilon(u)(x) + \del PD^\epsilon_l(u)(x) \notag \\
&= \dfrac{2}{\epsilon^2} \int_{x-\epsilon}^{x+\epsilon} J(|y-x|/\epsilon)\left(f'(|y-x|S(y,x;u)^2) - f'(0) \right) S(y,x;u) dy \notag \\
&= \dfrac{2}{\epsilon^2} \int_{x-\epsilon}^{x+\epsilon} \left[ f''(0) p^3 |y-x| e \right. \notag \\
& \qquad \qquad \left. + (f''(0) p^2 q3/2 + f'''(0)p^5e/2) |y-x|^2 + T_2(y-x) \right] J(|y-x|/\epsilon) dy \notag \\
&= \dfrac{2}{\epsilon^2} \int_{x-\epsilon}^{x+\epsilon} f''(0) p^2 q3/2 |y-x|^2 J(|y-x|/\epsilon) dy + O(\epsilon^2) \notag \\
&= O(\epsilon),
\end{align*}
where terms with $e$ integrate to zero. From this, we see that same estimate holds when $u$ has continuous and bounded third or fourth derivatives. This proves the assertion of \cref{prop:2.1}. 

To prove \cref{eq:linear elastic force diff}, we proceed as follows
\begin{align*}
-\del PD^\epsilon_l(u)(x) &=  \dfrac{2}{\epsilon^2} \int_{x-\epsilon}^{x+\epsilon} J(|y-x|/\epsilon) f'(0)S(y,x;u)dy \notag \\
&= \dfrac{2}{\epsilon^2} \int_{x-\epsilon}^{x+\epsilon} J(|y-x|/\epsilon) f'(0) \left[ pe + q|y-x|/2 + T_1(y - x)/|y-x| \right] dy \notag \\
&= \left(\dfrac{2}{\epsilon^2} \int_{x-\epsilon}^{x+\epsilon} J(|y-x|/\epsilon) f'(0) |y-x|/2 dy \right) q + O(\epsilon) \notag \\
&= \bbC u_{xx}(x) + O(\epsilon),
\end{align*}
where we identify $\bbC$ using \cref{eq:elasto const C}, and $q = u_{xx}(x)$. This proves \cref{eq:linear elastic force diff}. 

To prove \cref{eq:linear to elasto 1}, we assume $u\in C^4(D)$ and \cref{eq:assump on u bd4}. Then, by Taylor series expansion, we have
\begin{align*}
S(y,x;u) &= u_x(x) \dfrac{y-x}{|y-x|} + 1/2u_{xx}(x) |y-x| \notag \\
&\qquad + 1/6 u_{xxx}(x) |y-x| (y-x) + 1/24 u_{xxxx}(\xi) |y-x|^3  \notag \\
&= pe + q|y-x|/2 + r |y-x|^2 e + T_1(y-x)/|y-x| ,
\end{align*}
where $T_1(y-x) = O(|y-x|^4)$. Substituting this into $-\del PD^\epsilon_l(u)(x)$, and noting that terms with $e$ integrate to zero, we get
\begin{align*}
-\del PD^\epsilon_l(u)(x) &=\bbC u_{xx}(x) + O(\epsilon^2).
\end{align*}

\subsection{Convergence of solution of peridynamic equation to the elastodynamic equation}
To prove \cref{thm:conv peri elasti}, we proceed as follows. Let $u^\epsilon$ be the solution of peridynamic model in \cref{eq:peri}, and let $u$ be the solution of elastodynamic equation in \cref{eq:elasto}. Boundary conditions and initial conditions are same as described in \cref{s:intro}. Assuming that the hypothesis of \cref{thm:conv peri elasti} holds, we have from \cref{prop:2.1}
\begin{align*}
-\del PD^\epsilon(u^\epsilon(t))(x) = \bbC u^\epsilon_{xx}(t,x) + O(\epsilon).
\end{align*}
We have also assumed that there exists $C_1 < \infty$ such that
\begin{align*}
\sup_{\epsilon >0} \left[ \sup_{(x,t)\in D \times J} | u^\epsilon_{xxxx}(t,x) | \right] < C_1 < \infty.
\end{align*}
Combining this together with \cref{eq:nonlinear elastic force diff} we have,
\begin{align*}
\sup_{(x,t)\in D\times J} |-\del PD^\epsilon(u^\epsilon(t))(x) - \bbC u^\epsilon_{xx}(t,x)| \leq C_3 \epsilon,
\end{align*}
where $C_3$ is independent of $x$, $t$ and $\epsilon$. Subtracting equation \cref{eq:peri} from equation \cref{eq:elasto} shows that  $e^\epsilon = u^\epsilon - u$ satisfies
\begin{align}\label{eq:error e eps}
\rho \ddot{e}^\epsilon(t,x) &= \bbC e^\epsilon_{xx}(t,x) + (-\del PD^\epsilon(u^\epsilon(t))(x) - \bbC u^\epsilon_{xx}(t,x)) \notag \\
&= \bbC e^\epsilon_{xx}(t,x) + F(t,x),
\end{align}
where 
\begin{align*}
F(t,x) &= -\del PD^\epsilon(u^\epsilon(t))(x) - \bbC u^\epsilon_{xx}(t,x) \,\,\,\,\,\hbox{  and  }\sup_{(x,t)\in D\times J} |F(t,x)|\leq  C_3\epsilon,
\end{align*}
with boundary condition and initial condition given by
\begin{align*}
e^\epsilon(0,x) &= 0, \qquad \dot{e}^\epsilon(0,x) = 0 \qquad \qquad \forall x\in D, \notag \\
e^\epsilon(t,x) &= 0, \qquad \dot{e}^\epsilon(t,x) = 0 \qquad \qquad \forall (t,x) \in [0,T] \times \partial D^\epsilon.
\end{align*}

Since $e^\epsilon$ satisfies \cref{eq:error e eps} we can apply Gronwall's inequality to  find
\begin{align}\label{eq:Gronwall}
\sup_{t\in J} \int_D \rho |\dot{e}^\epsilon(t,x)|^2dx + \int_D \bbC |e^\epsilon_x(t,x)|^2 dx \leq C_2 \epsilon^2.
\end{align}
Now to show that $e^\epsilon \to 0$ in $H^1(D)$, we apply \cref{eq:Gronwall} together with Poincare's inequality to get
\begin{align*}
||e^\epsilon(t,x)||^2_{L^2(D)} &\leq C ||e^\epsilon_x(t,x)||_{L^2(D)}^2 \notag \\
&\leq \frac{C}{\mathbb{C}} \sup_{t\in [0,T]} \leftcr \int_{D} \rho |\dot{e}^\epsilon(t,x)|^2 dx + \int_{D} \bbC |e^\epsilon_x(t,x)|^2 dx \rightcr \notag \\
&\leq \frac{C}{\mathbb{C}}C_2 \epsilon^2,
\end{align*}
where $C$ is the Poincare constant. On collecting results this shows that $e^\epsilon \to 0$ in the $H^1(D)$ norm with the rate $\epsilon$. This completes the proof of \cref{thm:conv peri elasti}. Identical arguments using \cref{eq:linear to elasto 1} deliver \cref{thm:conv linperi elasti}.

\subsection{Bounds on the consistency error}
We first prove for linear continuous interpolation and then extend the proof to higher order interpolations.

\subsubsection{Linear interpolation}
In this section \cref{prop:3.1} is established. We begin by writing the difference $S(y,x_i;\lininter{u})-S(y,x_i;u)$.
It is given by
\begin{equation}\label{eq:diff strain simple}
S(y,x_i;\lininter{u})-S(y,x_i;u)=\frac{\lininter{u}(y)-u(y)}{|y-x_i|}.
\end{equation}
From the hypothesis of \cref{prop:3.1} there is a constant $C$ for which $|u_{xx}|<C$ on $D$. Using the approximation property  $|\lininter{u}-u|\leq Ch^2$ and applying $|y-x_i|>h$ for $y$ outside the interval $[x_{i-1},x_{i+1}]$ gives
\begin{align*}
|S(y,x_i;\lininter{u})-S(y,x_i;u)|\leq \begin{cases}
C|y-x_i|& \qquad \text{if } y\in [x_{i-1}, x_{i+1}], \\
Ch& \qquad \text{if } y\in [x_{i}-\epsilon,x_{i-1}], \\
Ch& \qquad \text{if } y\in [x_{i+1}, x_{i}+\epsilon].
\end{cases}
\end{align*}
Note further that
$|y-x_i|\leq h$ for $y\in [x_{i-1}, x_{i+1}]$
and we conclude
\begin{equation}
\label{eq: interpstrain est}
|S(y,x_i;\lininter{u})-S(y,x_i;u)|\leq Ch.
\end{equation}
Straight forward calculation shows
\begin{align*}
|\nabla PD^\epsilon_\ell(\lininter{u})-\nabla PD^\epsilon_\ell(u)|&\leq\
\frac{2f'(0)}{\epsilon^2}\int_{x_i-\epsilon}^{x_i+\epsilon}|S(y,x_i,\lininter{u})-S(y,x_i;u)|J(|y-x_i|/\epsilon)\,dy\\
&\leq \frac{4f'(0)MCh}{\epsilon},\notag
\end{align*}
where $M=\max_{0\leq z<1} J(z)$ and \cref{eq:claim 2} of  \cref{prop:3.1} follows.

We now establish the consistency error for the nonlinear nonlocal model. We begin with an estimate for the strain.
Applying the notation described in \cref{eq:notat 1} with $p$ and $e$  defined for $x=x_i$ we apply Taylor's theorem with reminder  to get
\begin{align}\label{eq:taylor S Ctwo}
S(y,x_i;u) &= u_x(x_i) (y-x_i)/|y-x_i| + u_{xx}(\xi)|y-x_i|/2 \notag \\
&= p e + T_1(y-x_i)/|y-x_i|
\end{align}
where $T_1(y-x_i) = O(|y-x_i|^2)$.

From \cref{eq: interpstrain est} we can write
\begin{equation*}
S(y,x_i;u)=S(y,x_i;\lininter{u})+O(h),
\end{equation*}
or
\begin{equation}
\label{eq: interpstrain estorderflip}
S(y,x_i;\lininter{u})=S(y,x_i;u)+O(h),
\end{equation}
or we can write  $S(y,x_i;u)=S(y,x_i;\lininter{u})+\eta$, where $|\eta|<Ch$. Adopting this  convention first we write
\begin{align*}
&|y-x_i|S^2(y,x_i;u)=|y-x_i|(S(y,x_i;\lininter{u})+\eta)^2\\
&=|y-x_i|S^2(y,x_i;\lininter{u})+\sum_{j\in I}2\eta\phi_j(y)(u(x_j)-u(x_i))+|y-x_i|\eta^2,\notag
\end{align*}
where the set $I=\{j\,:x_j\in[x_i-\epsilon,x_i+\epsilon]\} $ and we have used the identity
\begin{align*}
1=\sum_{j\in I}\phi_j(y), y\in[x_{i}-\epsilon,x_i+\epsilon].
\end{align*}
Next we estimate
\begin{align*}
\sum_{j\in I}\phi_j(y)(u(x_j)-u(x_i))&\leq\sup\{|u(y)-u(x_i)| \hbox{  for $y\in[x_{i}-\epsilon,x_i+\epsilon]$}\}\\
&\leq2 \epsilon\sup_{y\in D}|u_x(y)|.\notag
\end{align*}
Since $u_x$ is bounded we see that
$\sum_{j\in I}\phi_j(y)(u(x_j)-u(x_i))=\zeta$ where $|\zeta|\leq\epsilon \,Const.$ and
\begin{align*}
&|y-x_i|S^2(y,x_i;u)=|y-x_i|(S(y,x_i;\lininter{u})^2+2\zeta\eta+\eta^2
\end{align*}
Applying Taylor's theorem with remainder to the function $f'(|y-x|(S(y,x_i;\lininter{u})+\eta)^2)$ now gives
\begin{align}\label{eq:taylor fprime Ctwo}
f'(|y-x|S(y,x_i;u)^2) &=f'(|y-x|(S(y,x_i;\lininter{u}))^2) \\
&+O(h),\notag
\end{align}
where we have used that $f''(r)$ is bounded on $0\leq r\leq \infty$.

Then application of \cref{eq:taylor S Ctwo},  \cref{eq: interpstrain estorderflip}, and \cref{eq:taylor fprime Ctwo} and substitution delivers the desired estimate

\begin{align*}
&-\del PD^\epsilon(\lininter{u})(x_i)  +\del PD^\epsilon(u)(x_i)\notag \\
&= \dfrac{2}{\epsilon^2} \int_{x_i-\epsilon}^{x_i + \epsilon} f'(|y-x|(S(y,x_i;\lininter{u}))^2)\left(S(y,x_i;{u})+O(h)\right)J(|y-x_i|/\epsilon) dy  \\
&-\dfrac{2}{\epsilon^2} \int_{x_i-\epsilon}^{x_i + \epsilon} 
\left(f'(|y-x|(S(y,x_i;\lininter{u}))^2)+O(h)\right)S(y,x_i;{u}) J(|y-x_i|/\epsilon) dy \notag \\
&= O(h/\epsilon),
\end{align*}

and \cref{eq:claim 2nonlin} of \cref{prop:3.1} is proved.

\subsubsection{Higher order interpolations and convergence}
In this section we outline the proof of higher order accuracy using higher order interpolation functions when the solution has sufficiently high order bounded derivatives. The order of the interpolation is $p$, the mesh size $h$, and the grid points are $x_i=ih/p$ for $i\in K\cup K^\epsilon$. We state the following key result:

\begin{lemma}\label{pinterp}
If $u\in C^{p+1}(D)$ with $(p+1)^{\text{th}}$ derivative bounded then we have for $p^{th}$ order interpolation the following estimate
\begin{align}
|S(y,x_i;\lininter{u}) - S(y,x_i; u)| \leq \tilde{C} h^p \qquad \forall i\in K, \forall y \in [x_i - \epsilon, x_i+\epsilon]
\end{align}
where constant $\tilde{C}$ is independent of $h$, $i$, and $y$.
\end{lemma}

\begin{proof}
Fix some $i\in K$. There exist $C>0$ such that $\sup |\partial_{x}^{p+1} u| < C$. The interpolation error \cite{IK} is $|\lininter{u}(y) - u(y)| \leq C h^{p+1}$ for all $y \in \overline{D\cup \partial D^\epsilon}$. Now for $y\in [x_i - \epsilon, x_{i-1}] \cup [x_{i+1}, x_i +\epsilon]$, $h \leq |y-x_i|$ and hence $\frac{1}{|y-x_i|} \leq \frac{1}{h}$. Thus, we have from \cref{eq:diff strain simple}
\begin{align}\label{eq:strain estimate 1}
|S(y,x_i;\lininter{u}) - S(y,x_i; u)| \leq C h^p.
\end{align}
\end{proof}

The proofs of \cref{prop:3.3} and \cref{prop:3.4} now follow using \cref{pinterp} and applying the same steps used in the proof of \cref{prop:3.1} and \cref{prop:3.2} for linear interpolation.

\section{Conclusion}\label{s:concl}
Earlier related work \cite{CMPer-JhaLipton} analyzed the model considered here but for less regular non-differentiable H\"older continuous solutions. For that case solutions can approach discontinuous deformations (fracture like solutions) as $\epsilon\rightarrow 0$ and it is shown that the numerical approximation of the nonlinear model in dimension $d=1,2,3$ converges to the exact solution at the rate $O(\Delta t + h^\gamma/\epsilon^2)$ where $\gamma\in (0,1]$ is the H\"older exponent, $h$ is the size of mesh, $\epsilon$ is the size of horizon, and $\Delta t$ is the size of time step.  In this work we have shown that we can improve the rate of convergence if we somehow have a-priori knowledge on the number of bounded continuous derivatives of the solution. If the solution has $p+1$ derivatives one can use $p^{th}$ order polynomial local interpolation and obtain an  order $h^p/\epsilon$ consistency error.

In this work we have analyzed  the smooth prototypical micro-elastic bond model introduced in \cite{CMPer-Silling}.
From the perspective of computation, the resolution of the mesh inside the horizon of nonlocal interaction is the main contributor to the computational complexity. This work provides explicit error estimates for the differences between the solutions of elastodynamics and nonlocal models. It shows that the effects of the mesh size relative to the horizon can be significant. Numerical errors can grow with decreasing horizon if the mesh is not chosen suitably small with respect to the peridynamic horizon.  A fixed ratio of mesh size to horizon will not increase accuracy as the horizon tends to zero. We have carried out numerical simulations where the accuracy decreases when $\epsilon$ is reduced and the ratio $h/\epsilon$ is fixed. This is shown to be in line with the consistency error bounds that vanish at the rate $O(h/\epsilon)$.  These results show that the grid refinement relative to the horizon length scale has
more importance than decreasing the horizon length when establishing convergence to the
classical elastodynamics description.

The results of this analysis rigorously show that one can use a discrete linear local elastodynamic model to approximate the nonlinear nonlocal evolution when sufficient regularity of the evolution is known a-priori. In doing so one incurs a modeling error of order $\epsilon$ but saves computational work in that there is no nonlocality so the mesh diameter $h$ no longer has to be small relative to $\epsilon$. The discretization error is now associated with the approximation error for the initial boundary value problem for the linear elastic wave equation. 

We reiterate that the nonlinear kernel analyzed here corresponds to a smooth version of the prototypical micro-elastic bond model treated in \cite{CMPer-Silling}. On the other hand  its linearization corresponds to the one of the types  kernel functions treated in \cite{ChenBakenhusBobaru}. In this paper the goal is to understand the convergence of numerical schemes for the nonlinear model together with its linearization with respect to horizon and discretization. The work of \cite{ChenBakenhusBobaru} asks distinctly different questions and is concerned with identifying linear nonlocal models that converge to linear elastodynamics when the mesh density is held fixed and the horizon of nonlocality goes to zero. This is not the case for the kernel treated here.

Our results and analysis support a combined local - nonlocal approach to the numerical solution of these problems. This type of numerical approach is the focus of many recent investigations, see \cite{CMPer-Wildman}, \cite{CMPer-Seleson}, \cite{CMPer-Kilic}, \cite{CMPer-Oterkus}, \cite{CMPer-Zaccariotto}, \cite{CMPer-Han}, \cite{CMPer-Lubineau}, and \cite{CMPer-Liu}, where the use of nonlocal models and local models are applied to different subdomains of the computational domain. These approaches are promising in that they reduce the  computational cost of the numerical simulation. A full understanding of the error associated in implementing these adaptive methods is an exciting prospect for future research.

\section*{Acknowledgements}
RL would like to acknowledge the support and kind hospitality of the Hausdorff Institute for Mathematics in Bonn during the trimester program on multiscale problems.


\newcommand{\noopsort}[1]{}

\end{document}